\documentclass[leqno]{amsart}
\usepackage{amsfonts,amssymb,amsmath,amsgen,amsthm,mathtools}
\usepackage{hyperref}
\usepackage{color}
\usepackage{epsfig,epstopdf,xcolor,bm}
\usepackage{graphicx,wrapfig}
\usepackage{subfigure}
\usepackage{indentfirst}
\newcommand{\msc}[2][2000]{%
	\let\@oldtitle\@title%
	\gdef\@title{\@oldtitle\footnotetext{#1 \emph{Mathematics subject
				classification.} #2}}%
}

\theoremstyle{plain}
\newtheorem{theorem}{Theorem} [section]

\newtheorem{lemma}[theorem]{Lemma}

\newtheorem{proposition}[theorem]{Proposition}

\theoremstyle{remark}
\newtheorem{remark}[theorem]{Remark}

\def\({\left(}
\def\){\right)}

\newcommand{\wh}{\widehat}
\newcommand{\la}{\langle}
\newcommand{\ra}{\rangle}

\def\le{\leqslant}
\def\ge{\geqslant}

\def\<{\left\langle}
\def\>{\right\rangle}

\def\Re{\operatorname{Re}}
\def\Im{\operatorname{Im}}
\newcommand{\ov}{\overline}

\newcommand{\p}{\partial}

\newcommand{\be}{\begin{equation}}
	\newcommand{\ee}{\end{equation}}

\numberwithin{equation}{section}

\begin{document}
	\title[Filtered LIE-TROTTER Splitting for GB]{Filtered Lie-Trotter splitting for the ``good" Boussinesq equation: low regularity error estimates}
	
	\author[L. Ji]{Lun Ji}
	\author[H. Li]{Hang Li}
	\author[A. Ostermann]{Alexander Ostermann}
	\author[C. Su]{Chunmei Su}
	
	\address{Department of Mathematics\\Universit\"at Innsbruck\\Technikerstr.~13\\6020 Innsbruck\\Austria (L. Ji)}
	\email{lun.ji@uibk.ac.at}
	
	\address{Yau Mathematical Sciences Center\\Tsinghua University\\
		100084 Beijing\\China (H. Li)}
	\email{lihang20@mails.tsinghua.edu.cn}

	\address{Department of Mathematics\\Universit\"at Innsbruck\\Technikerstr.~13\\6020 Innsbruck\\Austria (A. Ostermann)}
	\email{alexander.ostermann@uibk.ac.at}
	
	\address{Yau Mathematical Sciences Center\\Tsinghua University\\
		100084 Beijing\\ China (C. Su)}
	\email{sucm@tsinghua.edu.cn}

	\begin{abstract}
		We investigate a filtered Lie-Trotter splitting scheme for the ``good" Boussinesq equation and derive an error estimate for initial data with very low regularity. Through the use of  discrete Bourgain spaces, our analysis extends to initial data in $H^{s}$ for $0<s\leq 2$, overcoming the constraint of $s>1/2$ imposed by the bilinear estimate in smooth Sobolev spaces. We establish convergence rates of order $\tau^{s/2}$ in $L^2$ for such levels of regularity. Our analytical findings are supported by numerical experiments.
	\end{abstract}
	\thanks{H. L. and C. S. are supported by National Key R\&D Program of China (Grant No. 2023YFA1008902) and the NSFC 12201342. }


	\maketitle
	
	\allowdisplaybreaks
	
	\section{Introduction}\label{intro}
	
	We consider the following periodic boundary value problem of the ``good" Boussinesq (GB) equation:
	\be\label{GB}
	\left\{	\begin{aligned}
		&z_{tt}+z_{xxxx}-z_{xx}-(z^2)_{xx}=0, \quad x \in \mathbb{T}=[-\pi,\pi], \quad t>0,\\
		&z(0, x)=\phi_0(x), \quad z_{t}(0, x)=\psi_0(x),
	\end{aligned}
	\right.\ee
	where $\phi_0$ and $\psi_0$ are given initial data. The GB equation, originally formulated by Boussinesq in 1872 \cite{Boussinesq1872}, is a fundamental model for describing the propagation of dispersive shallow water waves. This equation effectively captures the intricate interplay between various physical factors, including dispersion, nonlinearity, and wave-breaking effects in shallow water systems. Indeed, it has found wide-ranging applications in diverse fields such as oceanography \cite{tatlock2018assessment}, coastal engineering \cite{lakhan2003advances}, and physics \cite{kirby1996nonlinear,johnson1997modern}.
	
	In recent decades, extensive analytical and numerical research has been carried out on the GB equation. Studies, including \cite{farah2010periodic,farah2009local,kishimoto2013sharp,oh2013improved,wang2013well},
	have investigated the local well-posedness of \eqref{GB} on the torus $\mathbb{T}$. Specifically, for any initial data $(\phi_0, \psi_0)\in H^{s}(\mathbb{T})\times H^{s-2}(\mathbb{T})$ with $s\geq-1/2$, there exists a positive number $T_0=T_0(\lVert \phi_0\rVert_{H^s}, \lVert \psi_0\rVert_{H^{s-2}})>0$, and a unique mild solution $z\in C([0,T_0], H^s)$ to \eqref{GB}. For studies on the existence of traveling wave solutions, we refer to \cite{el2003numerical,de1991pseudospectral,lambert1987soliton,manoranjan1984numerical,manoranjan1988soliton}, and the references therein. In terms of numerical aspects, various conventional methods have been employed to solve the GB equation, including finite difference methods \cite{bratsos2007second}, finite element methods \cite{manoranjan1984numerical}, spectral methods \cite{cheng2015fourier, de1991pseudospectral,su2020deuflhard}, exponential integrators \cite{su2020deuflhard} and splitting methods \cite{zhang2017second}. These methods typically require very smooth solutions to achieve their classical convergence rates. Unfortunately, in practical applications, initial data often exhibit low regularity due to measurement errors or inherent noise. Unlike parabolic equations, where irregularities tend to smooth out over time, in dispersive equations such as the GB equation, rough initial data can propagate. This emphasizes the demand for the development of numerical methods that maintain their accuracy even when the data exhibit less regularity.
	
	Recently, Ostermann \& Su \cite{ostermann2019two} have proposed two low regularity exponential-type integrators (LREIs) for the GB equation, which can achieve first- and second-order accuracy in $H^r$ $(r>1/2)$ when the solutions are in $H^{r+1}$ and $H^{r+3}$, respectively. This requirement is significantly lower compared to methods such as the operator splitting approach \cite{zhang2017second} and the spectral method \cite{cheng2015fourier}, which need at least four additional derivatives to achieve second-order convergence in time. Building on this work, Li \& Su \cite{li2022lowregularity} developed new LREIs that outperform the ones proposed by Ostermann \& Su in terms of derivative requirements. In particular, their second-order LREI has quadratic convergence with only two additional derivatives required, thus relaxing the previously established regularity assumptions that require three additional derivatives. It should be noted, however, that the convergence results mentioned above require the error to be measured in $H^s$ with $s > \frac{1}{2}$, as the error estimate is based on the bilinear estimate
	\[
	\|fg\|_{H^s} \lesssim_s \|f\|_{H^s} \|g\|_{H^s},
	\]
	which is \emph{only} valid for $s > \frac{1}{2}$. This suggests that classical LREIs achieve the corresponding convergence only if the solution is bounded at least in $H^{s}$ with $s>1/2$. To address the challenge of handling low regularity initial data, in this work, we introduce an approach for the GB equation that builds on the theory of discrete Bourgain spaces, originally established by Ostermann, Rousset \& Schratz \cite{Ostermann2022JEMS}. We develop and analyze a filtered Lie-Trotter splitting (FLTS) scheme, which yields an error of $\mathcal{O}(\tau^{s/2})$ for a solution in $H^{s}$, where $0<s\le 2$ and $\tau$ is the time step size. The introduction of this FLTS is motivated by the fact that filtered schemes, as demonstrated in \cite{Ostermann2022JEMS,ji2023Lowregularity}, have been shown to maintain excellent convergence properties even under very low regularity initial data.
	
	The FLTS scheme is constructed and analyzed using the following strategies.
	
	\noindent (i) \emph{Construction of the scheme}. First, we introduce a complex variable
	\be\label{rel}
	u=z+1/2-i\langle \partial_x^2\rangle^{-1}z_t,\ee
	where $\langle \p_x^2\rangle:=(1+\p_x^4)^{1/2}$. Then the GB equation can be transformed into an equivalent first-order equation
	\be\label{orieq}
	\left\{	\begin{aligned}
		&\p_tu=i\la \p_x^2 \ra u-\frac{i}{4}\la \p_x^2 \ra^{-1}\left[2(u+\ov{u})+\p_x^2(u+\ov{u})^2\right],    \\
		&u(0, x)=u_0(x)=\phi_0(x)+\frac12-i\la \p_x^2 \ra^{-1}\psi_0(x).
	\end{aligned}
	\right.\ee
	Its projected form reads as follows
	\be\label{truneq}
	\left\{	\begin{aligned}
		&\p_tu_{\tau}=i\la \p_x^2 \ra u_{\tau}-\frac{i}{4}\la \p_x^2 \ra^{-1}\Pi_{\tau}\left[2(u_{\tau}+\ov{u}_{\tau})+\p_x^2(\Pi_{\tau}u_{\tau}+\Pi_{\tau}\ov{u}_{\tau})^2\right],    \\
		&u_{\tau}(0, x)=\Pi_{\tau}\phi_0(x)+1/2-i\la \p_x^2 \ra^{-1}\Pi_{\tau}\psi_0(x),
	\end{aligned}
	\right.\ee
	where the projection operator $\Pi_{\tau}$ for $\tau>0$ is defined by the Fourier multiplier
	\[
	\Pi_{\tau}=\chi\left(\frac{-i\partial_x}{\tau^{-1/2}}\right),
	\]
	with $\chi$ denoting the characteristic function of $[-1, 1]$. Then \eqref{truneq} is split into two subproblems that are easier to solve. By solving each subproblem, we finally obtain the Lie-Trotter splitting scheme for \eqref{truneq} (see \eqref{numflow} below), which is in fact the FLTS scheme for \eqref{orieq}.
	
	\noindent(ii) \emph{Well-posedness and error of the intermediate model \eqref{truneq}.} To establish the numerical error of the FLTS scheme for \eqref{orieq}, we first analyze the difference between \eqref{orieq} and its projected counterpart \eqref{truneq}. Unlike the globally well-posed Schr\"odinger equation \cite{Ostermann2022splitting}, the GB equation is only locally well-posed for general initial data \cite{oh2013improved}. Thus, it is necessary to investigate and compare the existence time of the two models. With the help of Bourgain spaces, we will rigorously prove that the solution of the projected equation \eqref{truneq} exists for a sufficiently long time if $\tau$ is small enough. This is crucial for establishing the difference between the two models \eqref{orieq} and \eqref{truneq}:
	\[\|u_{\tau}(t_n)-u(t_n)\|_{L^2} \le C_T\tau^{s/2},\quad 0<t_n\le T\ ({\rm cf. \ Prop.} \ \ref{wellposednessforpj}),\quad \mathrm{for}\quad u_0\in H^s.\]
	
	\noindent (iii) \emph{Error of the FLTS scheme for \eqref{orieq}.} By combining this estimate with the time discretization error of the FLTS scheme \eqref{numflow} for the projected equation \eqref{truneq}:
	\[
	\|u_{\tau}(t_n)-u^n\|_{L^2} \le C_T\tau^{s/2}, \quad 0<t_n\le T\ ({\rm cf.\ Sect.} \ \ref{global}), \quad \mathrm{for}\quad u_0\in H^s,
	\]
	we are able to establish the desired global error of $u^n$:
	\[
	\|u(t_n)-u^n\|_{L^2} \le C_T\tau^{s/2}, \quad 0<t_n\le T, \quad \mathrm{for}\quad u_0\in H^s.
	\]
	
	In view of the relation \eqref{rel}, we obtain the global error of the FLTS scheme.
	\begin{theorem}\label{th1}
		For $\phi_0\in H^{s}$ and $\psi_0\in H^{s-2}$, let $T_s\in(0, \infty]$ be the maximum existence time of the solution of \eqref{GB}. Furthermore, for any $T<T_{s}$, $s\in(0, 2]$, we denote the unique solution of \eqref{GB} with initial data $(\phi_0, \psi_0)$ as $(z, z_t) \in C([0, T], H^{s})\times C([0, T], H^{s-2})$, and the numerical solution \eqref{znscheme} as $(z^n, z_t^n)$. Then, there exists $\tau_0>0$ such that for $\tau\le\tau_0$, the error is bounded as follows:
\[ \|z(t_n)-z^n\|_{L^2}+\|z_t(t_n)-z_t^n\|_{H^{-2}}\lesssim_{T_{}} \tau^{s/2},\quad 0<t_n=n\tau\le T.
		\]
	\end{theorem}
		\begin{remark}
The focus of this paper is exclusively on time discretization. A full discretization employing a filtered Lie splitting scheme in time, in conjunction with a pseudospectral method in space, has been previously analyzed for the nonlinear Schr\"odinger equation \cite{ji2024Lowregularityfull}. By applying similar techniques within the framework of Bourgain spaces, we can establish the full discrete scheme along with its corresponding convergence for the GB equation. Details are omitted here for the sake of brevity.
	\end{remark}
	
	The rest of the paper is structured as follows. In Section \ref{LT}, we derive the Lie-Trotter splitting scheme for the Schr\"odinger-type equation \eqref{orieq}, which is equivalent to the original GB equation \eqref{GB}. Section \ref{LWP} provides proofs for the local existence and uniqueness of solutions for both \eqref{orieq} and its projected version \eqref{truneq}, along with the estimates for the differences between their solutions. Section \ref{DBS} introduces the properties of discrete Bourgain spaces, which are then utilized in Section \ref{LRLT} to provide global error estimates for the FLTS scheme. Finally, in Section \ref{num}, we present the results of numerical experiments.
	\vspace{0.27cm}
	
	\noindent\textbf{Notations.} In this paper, the notation $X \lesssim Y$ indicates that there exists a constant $C>0$, which may vary from line to line but remains independent of the time step $\tau$, such that $|X|\le C Y$. In addition, $X\sim Y$ means $Y\lesssim X\lesssim Y$. We use the notation $\lesssim_{\gamma}$ to emphasize that $C$ depends on $\gamma$ and we denote $\langle \,\cdot\,\rangle =\sqrt{1+|\cdot|^2}$. For a sequence of functions $(u_n)_{n\in \mathbb{Z}}$ with each $u_n$ in some Banach space $B$, we employ the norms
	\[
	\|u_n\|_{l_{\tau}^pB}=\Big(\tau\sum\limits_n\|u_n\|_{B}^p\Big)^{\frac{1}{p}},
	\quad \|u_n\|_{l_{\tau}^{\infty}B}=\sup\limits_{n\in \mathbb{Z}}\|u_n\|_{B}.
	\]
	

	\section{A first-order FLTS scheme for GB equation}\label{LT}
	In this section, we  derive the filtered Lie-Trotter splitting scheme based on an equivalent formulation of the GB equation.
	\subsection{Reformulation of the GB equation}\label{preGB}
	By using the approach outlined in \cite{ostermann2019two} and defining $\check{z}=z+\frac{1}{2}$, we can rephrase the  GB equation as follows
	\[
	\left\{	\begin{aligned}
		&\check{z}_{tt}+\check{z}_{xxxx}-(\check{z}^2)_{xx}=0, \quad x \in \mathbb{T}, \quad t>0,\\
		&\check{z}(0, x)=\phi_0(x)+\frac{1}{2}, \quad \check{z}_{t}(0, x)=\psi_0(x).
	\end{aligned}
	\right. \]
	This can be further reformulated as a system
	\be\label{GBcoup}
	\left\{	\begin{aligned}
		&\p_tu=i\la \p_x^2 \ra u-\frac{i}{4}\la \p_x^2 \ra^{-1}\left[2(u+\ov{v})+\p_x^2(u+\ov{v})^2\right],    \\
		&\p_tv=i\la \p_x^2 \ra v-\frac{i}{4}\la \p_x^2 \ra^{-1}\left[2(\ov{u}+v)+\p_x^2(\ov{u}+v)^2\right],
	\end{aligned}
	\right.\ee
	where
	\[
	u=\check{z}-i\la \p_x^2 \ra^{-1}\check{z}_t, \quad v=\ov{\check{z}}-i\la \p_x^2 \ra^{-1}\ov{\check{z}_t}.
	\]
	Since $z\in \mathbb{R}$, we conclude $u=v$, and
	\[
	\check{z}=\frac{1}{2}(u+\ov{u}), \quad \check{z}_t=\frac{i}{2} \la \p_x^2 \ra^{} (u-\ov{u}).
	\]
	Thus the system \eqref{GBcoup} involving $u, v\in \mathbb{C}$ is reduced to the Schr{\" o}dinger-type equation \eqref{orieq},
	and the solution of \eqref{GB} can be recovered via
	\be\label{recoverz}
	z=\frac{1}{2}(u+\ov{u})-\frac{1}{2}, \quad z_t=\frac{i}{2} \la \p_x^2 \ra^{} (u-\ov{u}).
	\ee
	For the remainder, we will refer to both the original equation \eqref{GB} and the aforementioned equivalent first-order equation \eqref{orieq} as the GB equation. In the following, we will establish the Lie-Trotter splitting scheme based on the projected GB equation \eqref{truneq}, which is actually our FLTS scheme.
	\subsection{A Lie-Trotter splitting method for the projected GB equation}\label{LTGB}
	It is clear that \eqref{truneq} is exactly the projected counterpart of \eqref{orieq}. Now we construct a classical  Lie-Trotter splitting discretization of \eqref{truneq} based on the following splitting
	\[
	\partial_{t}u_{\tau} =X_1(u_{\tau})+X_2(u_{\tau}),
	\]
	where
	\[
	X_1 (u_{\tau})=i\langle \partial_x^2 \rangle u_{\tau}, \quad X_2(u_{\tau})=-\frac{i}{4}\langle \partial_x^2 \rangle^{-1}\Pi_{\tau}\left[2(u_{\tau} + \overline{u}_{\tau})+\p_x^2(\Pi_{\tau}u_{\tau}+ \Pi_{\tau}\ov{u}_{\tau})^2\right].
	\]
	The corresponding subproblems can be written as
	\[
	\left\{
	\begin{aligned}
		&\partial_{t}\nu (t,x)=X_1(\nu(t,x)),\\
		&\nu(0,x)=\nu_0(x), \quad x\in \mathbb{T}, \quad t>0,
	\end{aligned}\right.\quad\quad\quad
	\left\{
	\begin{aligned}
		&\partial_{t}\omega(t,x) =X_2(\omega(t,x)),\\
		&\omega(0,x)=\omega_0(x),  \quad x\in \mathbb{T}, \quad t>0.
	\end{aligned}\right.
	\]
	Notably, the solution to the linear subproblem can be written as
	\[
	\nu(t, \cdot)=\varPhi_{X_1}^t(\nu_0)={\rm e}^{it\langle \partial_x^2 \rangle}\nu_0,
	\]
	which allows for precise integration in phase space. In the case of the nonlinear subproblem, by writing $\omega=\omega_1+i\omega_2$ with $\omega_1, \omega_2\in \mathbb{R}$, we observe that
	\begin{align}\label{w121}
		\p_t\omega_1= 0,\quad
		\p_t\omega_2=-\langle \partial_x^2 \rangle^{-1}\Pi_{\tau}\left[\omega_1+\p_x^2(\Pi_{\tau}\omega_1)^2\right].
	\end{align}
	Therefore, we have
	\[
	\omega_1(t)=\omega_1(0)=\operatorname{Re} \omega_0,\quad
	\p_t\omega_2=-\langle \partial_x^2 \rangle^{-1}\Pi_{\tau}\left[\omega_1(0)+\partial_x^2(\Pi_{\tau}\omega_1(0))^2\right],
	\]
	which yields
	\[
	\omega_2(t)=\Im\omega_0-t\langle \partial_x^2 \rangle^{-1}\Pi_{\tau}\left[\omega_1(0)+\partial_x^2(\Pi_{\tau}\omega_1(0))^2\right],\]
	and
	\begin{align*} \omega(t)&=\omega_1(t)+i\omega_2(t)\\
		&=\Re\omega_0+i\left[\Im\omega_0-t\la \p_x^2\ra^{-1}\Pi_{\tau}\omega_1(0)-t\la\p_x^2\ra^{-1}\Pi_{\tau}\partial_x^2(\Pi_{\tau}\omega_1(0))^2 \right]	\\
		&=\omega_0-\frac{it}{2}\la \p_x^2\ra^{-1}\Pi_{\tau}\left[(\omega_0+\ov{\omega}_0)+\frac{1}{2}\partial_x^2(\Pi_{\tau}\omega_0+\Pi_{\tau}\ov{\omega}_0)^2\right]\\
		&\eqqcolon\varPhi_{X_2}^t(\omega_0).
	\end{align*}
	Finally, the filtered first-order Lie-Trotter splitting scheme for the GB equation takes the form
	\begin{subequations}\label{numflow}
		\begin{align}
			u^{n+1}_{}=\varPhi_{}^{\tau}(u^n_{})=\varPhi^{\tau}_{X_1}(\varPhi^{\tau}_{X_2}(u^n_{})), \quad u^0_{}=\Pi_{\tau}u_0,
		\end{align}
		where
		\begin{align}
			\varPhi^{\tau}(v)={\rm e}^{i\tau\la \p_x^2\ra}\Pi_{\tau}\Big[v-\frac{i\tau}{2}\la \p_x^2\ra^{-1}(v+\ov{v})-\frac{i\tau}{4}\la \p_x^2\ra^{-1}\partial_x^2(\Pi_{\tau}v+\Pi_{\tau}\ov{v})^2\Big],
		\end{align}
	\end{subequations}
	and the numerical approximations for $z(t_n)$ and $z_t(t_n)$ are given by
	\begin{align}\label{znscheme}
		z^n_{}=\frac{1}{2}(u^n_{}+\overline{u}^n_{})-\frac{1}{2},\quad
		z^n_{t}=\frac{i}{2}\langle\partial_x^2\rangle(u^n_{}-\overline{u}^n_{}).
	\end{align}
	Correspondingly, high-order splitting methods can be constructed.

	\medskip
	\section{Local Well-posedness of \eqref{orieq} and its projected equation}\label{LWP}
	In this section, we focus on studying the local existence and uniqueness for the GB equation \eqref{orieq}. This step is crucial for comparing the existence times of solutions for the original equation \eqref{orieq} and its projected version \eqref{truneq}. We begin by recalling the definition of Bourgain spaces \cite{Chapouto2018Fourier}.
	To do this, let us consider the space-time Fourier transform and its inverse:
	\begin{align*}
		\mathcal{F}(u)(\sigma, k)&=\wh{u}(\sigma, k)=\frac{1}{2\pi}\int_{\mathbb{R}}\int_{\mathbb{T}}u(t,x){\rm e}^{-i\sigma t-ikx}dxdt,\quad \sigma\in\mathbb{R},\quad k\in\mathbb{Z};\\
		u(t,x)&=\frac{1}{2\pi}\sum\limits_{k\in \mathbb{Z}}\int_{\mathbb{R}}\wh{u}(\sigma, k){\rm e}^{i\sigma t+ikx}d\sigma,\quad t\in\mathbb{R},\quad x\in \mathbb{T}.
	\end{align*}
	For a given continuous function $g: \mathbb{R}\rightarrow \mathbb{R}$, the Bourgain space $X^{s, b}_{\sigma=g(k)}(\mathbb{R}\times\mathbb{T})$ with $s$, $b\in\mathbb{R}$ is defined as the set of all tempered distributions that are bounded in the following norm:
	\begin{align}\label{norm}
		\|u\|_{X^{s,b}_{\sigma=g(k)}}:=\left\| \la k\ra^s\la \sigma -g(k)\ra^b \wh{u}(\sigma, k)\right\|_{l^2_kL_{\sigma}^2}.
	\end{align}
	It is evident that for any given continuous function $g$, $X^{0,0}_{\sigma=g(k)}=L^2(\mathbb{R}\times\mathbb{T})$ and
	\[
	\|u\|_{X^{0, 0}_{\sigma=g(k)}}=\left\|\wh{u}(\sigma, k)\right\|_{l^2_kL^2_{\sigma}}=\left\|u\right\|_{L^2(\mathbb{R}\times\mathbb{T})}.\]
	For a closed interval $I\subset \mathbb{R}$, we consider a localized Bourgain space denoted by $X_{\sigma=g(k)}^{s,b}(I)$, which consists of functions such that the norm $\|u\|_{X^{s,b}(I)}$ is finite, where
	\begin{align}\label{locali}
		\|u\|_{X^{s,b}_{\sigma=g(k)}(I)}:=\inf\{\|v\|_{X^{s,b}_{\sigma=g(k)}}, \ v|_{I}=u\}.
	\end{align}
	For brevity, we denote  $X^{s, b}_{\sigma=g(k)}([0, L])$ with $L>0$ by $X^{s, b}_{\sigma=g(k)}(L)$.
	
	In the case of the GB equation, following its formulation, we define $g(k)=\sqrt{k^4+1}$ and abbreviate the corresponding Bourgain space $X^{s, b}_{\sigma=\sqrt{k^4+1}}(\mathbb{R}\times\mathbb{T})$ as $X^{s,b}$ for simplicity. Considering $k^2\le \sqrt{k^4+1}\le k^2+1$, we conclude that there exists $c>0$ such that  for all $\sigma\in\mathbb{R}$ and all $k\in \mathbb{Z}$, the inequality
	\[
	\frac{1}{c}\le \frac{1+|\sigma-k^2|}{1+|\sigma-\sqrt{k^4+1}|}\le c
	\]
	holds.
	This implies
	\be\label{equiva}
	\begin{split}
		\|u\|_{X^{s,b}}&=\left\| \la k\ra^s\la \sigma -\sqrt{k^4+1}\ra^b \wh{u}(\sigma, k)\right\|_{l^2_kL_{\sigma}^2}\\
		&\sim \left\| \la k\ra^s\la \sigma -k^2\ra^b \wh{u}(\sigma, k)\right\|_{l^2_kL_{\sigma}^2}=\|u\|_{X_{\sigma=k^2}^{s,b}},
	\end{split}
	\ee
	and $X^{s,b}=X^{s, b}_{\sigma=\sqrt{k^4+1}}=X^{s, b}_{\sigma=k^2}$.
	Applying the fact (cf. \cite[Proposition 2.2]{Chapouto2018Fourier})
	\be\label{bourelation}
	\|u\|_{X^{s, b}_{\sigma=g(k)}}=\|\overline{u}\|_{X^{s, b}_{\sigma=-g(-k)}}\ee
	and the properties of the Bourgain space ${X^{s, b}_{\sigma=-k^2}}$
	established in \cite{Chapouto2018Fourier,TTao}, we easily derive the following lemma.
	
	\begin{lemma}\label{contibour}
		For $s\in \mathbb{R}$, $u_0\in H^s$, $f\in X^{s,b}$ and $\eta\in C_c^{\infty}(\mathbb{R})$, we have that
		\begin{align}
			&\left\|f\right\|_{L^{\infty}(\mathbb{R}, H^s)}\lesssim_{b} \|f\|_{X^{s,b}}, \ \ \  b>\tfrac{1}{2}, \label{lemma11}\\
			&\left\|\eta(t)f\right\|_{X^{s,b}}\lesssim_{\eta, b}\left\|f\right\|_{X^{s,b}}, \ \ \ b \in \mathbb{R},\label{lemma12}\\
			&\left\|\eta(t){\rm e}^{it\la \p_x^2\ra}u_0\right\|_{X^{s,b}}\lesssim_{\eta, b} \|u_0\|_{H^s}, \ \ \ b \in \mathbb{R}, \label{lemma13}\\
			&\left\|\eta(t)\int_0^t{\rm e}^{i(t-s)\la \p_x^2\ra}F(s)ds\right\|_{X^{s,b}}\lesssim_{\eta, b}\left\|F\right\|_{X^{s, b-1}}, \quad  b>\tfrac{1}{2}, \label{lemma14}\\
			&\left\|\eta\left(t/T\right)f\right\|_{X^{s,b^{\prime}}}\lesssim_{\eta, b, b^{\prime}}T^{b-b^{\prime}}\|f\|_{X^{s,b}},  \quad -\tfrac{1}{2}< b^{\prime}\le b<\tfrac{1}{2}, \quad 0<T\le 1, \label{lemma15}\\
			&\left\|f-\Pi_{\tau}f\right\|_{X^{0,b}}\lesssim \tau^{s/2}\|f\|_{X^{s,b}}, \quad b\in \mathbb{R}, \label{lemma16}\\
			&\left\|\Pi_{\tau}f\right\|_{X^{s,b}(L)}\lesssim \tau^{-s/2}\|f\|_{X^{0,b}(L)}, \quad s\ge 0, \quad b\in \mathbb{R}, \quad L>0. \label{lemma17}
		\end{align}
	\end{lemma}
	\begin{proof} By utilizing \eqref{equiva}, \eqref{bourelation} and Propositions 2.3-2.6, Lemma 2.7 in \cite{Chapouto2018Fourier}, we immediately derive \eqref{lemma11}-\eqref{lemma15}.
		Furthermore, \eqref{lemma16}-\eqref{lemma17} follow from the definition of $\Pi_\tau$ straightforwardly and the proof is completed.
	\end{proof}
	
	Now, we recall the Kato-Ponce inequality, which was derived in \cite{Kato1988CommutatorEA,bourgain2014endpoint,li2019kato,LiWu22} for $\mathbb{R}$ or $\mathbb{T}$.
	\begin{lemma} (Kato-Ponce inequality)
		For $s>0$, $1<p< \infty$, $1<p_1, p_3< \infty$, $1<p_2, p_4\le \infty$ satisfying $\frac{1}{p}=\frac{1}{p_1}+\frac{1}{p_2}$ and $\frac{1}{p}=\frac{1}{p_3}+\frac{1}{p_4}$, we have the following inequality
		\be\label{kp1}
		\big\|\la\partial_x\ra^s (fg)\big\|_{L^p}\lesssim \big\|\la\partial_x\ra^sf\big\|_{L^{p_1}}\big\|g\big\|_{L^{p_2}}+
		\big\|\la\partial_x\ra^s g\big\|_{L^{p_3}}\big\|f\big\|_{L^{p_4}},
		\ee
		where $\la\partial_x\ra^s f=\sum\limits_{k\in\mathbb{Z}}\langle k\rangle^s \widehat{f}_k e^{ikx}$ for $f=\sum\limits_{k\in\mathbb{Z}}\widehat{f}_k e^{ikx}$.
	\end{lemma}
	
	Using the refined bilinear $L^4$ estimate, we are able to prove the following lemma, which will play an essential role in establishing the local well-posedness for the GB equation.
	\begin{lemma}\label{bilinear}
		For $s\ge 0$ and $u, v\in X^{s, \frac{3}{8}}$, we have
		\begin{align*}
			\|uv\|_{X^{s, 0}}\lesssim \|u\|_{X^{s,\frac{3}{8}}}\|v\|_{X^{s, \frac{3}{8}}}.
		\end{align*}
		The inequality also holds when replacing $uv$ with $u\ov{v}$ or $\ov{u}\ov{v}$.
	\end{lemma}
	\begin{proof}
		Recalling the refined bilinear $L^4$ estimate \cite[Lemma 3.1]{kishimoto2013sharp}:
		\begin{align*}
			\|uv\|_{L^{2}_{}(\mathbb{R}\times\mathbb{T})}\lesssim \|u\|_{X^{0,b}}\|v\|_{X^{0, b^{\prime}}},\quad
			\mathrm{for}\quad  b, b^{\prime}>\frac{1}{4} \quad \mathrm{satisfying}\quad b+b^{\prime}\ge \frac{3}{4},
		\end{align*}
		taking $v=u$ and $b=b^{\prime}=\frac{3}{8}$, we obtain
		\begin{align}\label{l4}
			\|u\|_{L^4_{}(\mathbb{R}\times\mathbb{T})}\lesssim \|u\|_{X^{0,\frac{3}{8}}}.
		\end{align}
		Employing \eqref{kp1} and \eqref{l4}, we are led to
		\begin{align*}
			\|uv\|_{X^{s, 0}}&=\|\la \partial_x \ra^s (uv)\|_{X^{0, 0}}=
			\left\|\la \partial_x\ra^s (uv)\right\|_{L^2(\mathbb{R}\times\mathbb{T})}\\
			&\lesssim \|\la \partial_x\ra^s u\|_{L^4(\mathbb{R}\times\mathbb{T})}\|v\|_{L^4(\mathbb{R}
				\times\mathbb{T})}+\|\la \partial_x\ra^sv\|_{L^4(\mathbb{R}\times\mathbb{T})}\|u\|_{L^4(\mathbb{R}\times\mathbb{T})}\notag\\
			&\lesssim \|u\|_{X^{s,\frac{3}{8}}}\|v\|_{X^{0,\frac{3}{8}}}+\|v\|_{X^{s,\frac{3}{8}}}\|u\|_{X^{0,\frac{3}{8}}}\notag\\
			&\lesssim \|u\|_{X^{s,\frac{3}{8}}}\|v\|_{X^{s,\frac{3}{8}}},
		\end{align*}
		which completes the proof.
	\end{proof}
	
	
	Next we present the result regarding the existence and uniqueness of solutions of the GB equation.
	\begin{proposition}\label{wellposed}
		For $u_0\in H^{s}$, $s\ge 0$ and $b\in(\frac{1}{2}, 1)$, there exists $T_s>0$ such that for any $T<T_s$ the GB equation \eqref{orieq} has a unique solution $u\in X^{s, b}(T)\subseteq C([0, T], H^{s})$.
	\end{proposition}
	\begin{proof}
		For $\delta>0$, we define the truncated functional
		\be\label{functional}
		\begin{split}
			F_{u_0}^{\delta}(v):=& \ \eta(t){\rm e}^{it\la \p_x^2\ra}u_0-\frac{i}{2}\la \p_x^2\ra^{-1}\eta(t)\int_0^t{\rm e}^{i(t-\theta)\la \p_x^2\ra}\eta\left(\frac{\theta}{\delta}\right)[v(\theta)+
			\ov{v}(\theta)]d\theta\\
			&-\frac{i}{4}\la \p_x^2\ra^{-1}\p_x^2 \;\eta(t)\int_0^t{\rm e}^{i(t-\theta)\la \p_x^2\ra}\eta\left(\frac{\theta}{\delta}\right)[v(\theta)+\ov{v}(\theta)]^2d\theta,
		\end{split}
		\ee
		where $\eta: \mathbb{R}\rightarrow[0,1]$ is a smooth nonnegative function that takes $1$ on $[-1, 1]$ and is supported in  $[-2, 2]$.
		Applying \eqref{lemma14}, \eqref{lemma15} and \eqref{lemma12}, we can bound the linear term by
		\begin{align*}
			\bigg\|\eta(t)\int_0^t&{\rm e}^{i(t-\theta)\la \p_x^2\ra}\eta\left(\frac{\theta}{\delta}\right)[v(\theta)+\ov{v}(\theta)]d\theta\bigg\|_{X^{s,b}}\lesssim \left\|\eta\left(\frac{\theta}{\delta}\right)[v(\theta)+\ov{v}(\theta)]\right\|_{X^{s,b-1}}\notag\\
			&\qquad\qquad\qquad\qquad\qquad\lesssim \delta^{1-b}\|v+\ov{v}\|_{X^{s,0}}\lesssim \delta^{1-b}\|v\|_{X^{s,0}}\lesssim \delta^{1-b}\|v\|_{X^{s,b}}.
		\end{align*}
		Similarly, the nonlinear term can be bounded by
		\begin{align*}
			\bigg\|\eta(t)\int_0^t&{\rm e}^{i(t-\theta)\la \p_x^2\ra}\eta\left(\frac{\theta}{\delta}\right)\Big(v(\theta)+\ov{v}(\theta)\Big)^2d\theta\bigg\|_{X^{s,b}}\lesssim \left\|\eta\left(\frac{t}{\delta}\right)\left(v(t)+\ov{v}(t)\right)^2\right\|_{X^{s,b-1}}\notag\\
			&\qquad\qquad\qquad\qquad\lesssim \delta^{1-b}\left\|\left(v+\ov{v}\right)^2\right\|_{X^{s,0}}\lesssim \delta^{1-b}\|v\|^2_{X^{s,\frac{3}{8}}}\lesssim \delta^{1-b}\|v\|^2_{X^{s,b}},
		\end{align*}
		where we have used Lemma \ref{bilinear} for the last second inequality. Therefore, we conclude
		\[
		\|F_{u_0}^{\delta}(v)\|_{X^{s,b}}\lesssim \|u_0\|_{H^{s}}+\delta^{1-b}(\|v\|_{X^{s,b}}+\|v\|_{X^{s,b}}^2).
		\]
		Similarly, for $v_1$ and $v_2$ lying in the closed ball $B(0, R_0)$ of $X^{s,b}$ with $R_0=R_0(\|u_0\|_{H^{s}})$, we get
		\begin{align}\label{stafun}
			\|F_{u_0}^{\delta}(v_1)-F_{u_0}^{\delta}(v_2)\|_{X^{s,b}}\le C\delta^{1-b}(1+R_0)\|v_1-v_2\|_{X^{s,b}}.
		\end{align}
		Thus, $F_{u_0}^\delta$ is a contraction on $B(0, R_0)$ if $\delta=\delta_1\in(0,1]$ is sufficiently small. By employing the fixed-point theorem, we get that there exists a fixed-point $v$ for $F_{u_0}^{\delta_1}$, i.e., $F_{u_0}^{\delta_1}(v)=v$.
		Therefore $v|_{[0,\delta_1]}=u\in X^{s, b}(\delta_1)$, where $u$ is the solution of \eqref{orieq}. For uniqueness, we refer to the similar proof presented in \cite{farah2009local}.
		By iterating this procedure and employing standard arguments, it can be de\-duced that the solution can be extended to $[0, T_s)$ with maximum time $T_s\in(0, \infty]$.
	\end{proof}
	
	Next we examine the well-posedness of the projected equation \eqref{truneq}. For this purpose, we define the truncated functional
	\be\label{truncatedfunctional}
	\begin{split}
		F_{u_0}^{\tau,\delta}(v)&:=\eta(t){\rm e}^{it\la \p_x^2\ra}\Pi_{\tau}u_0-\frac{i}{2}\la \p_x^2\ra^{-1}\eta(t)\int_0^t{\rm e}^{i(t-\theta)\la \p_x^2\ra}\eta\left(\frac{\theta}{\delta}\right)\Pi_{\tau}[v(\theta)+
		\ov{v}(\theta)]d\theta\\
		&\quad\,\,-\frac{i}{4}\la \p_x^2\ra^{-1}\p_x^2 \;\eta(t)\int_0^t{\rm e}^{i(t-\theta)\la \p_x^2\ra}\eta\left(\frac{\theta}{\delta}\right)\Pi_{\tau}[\Pi_{\tau}v(\theta)+\Pi_{\tau}\ov{v}(\theta)]^2d\theta.
	\end{split}
	\ee
	We will prove that for given initial data $u_0\in H^s$ and any $T<T_s$,
	where $T_s$ is the maximum existence time of the solution of \eqref{orieq}, the solution $u_{\tau}$ of the projected equation \eqref{truneq} exists on $[0, T]$ for $\tau$ sufficiently small. We will acomplish this by estimating the difference between $u$ and $u_{\tau}$.
	
	\begin{proposition}\label{wellposednessforpj}
		For $u_0\in H^{s}$ ($s\ge 0$) and $b\in(\frac{1}{2}, 1)$, let $[0,T_s)$ denote the maximum existence interval of \eqref{orieq}. Then, for any $T<T_{s}$, there exists $\tau_0>0$ such that for $\tau\le \tau_0$, there is a unique solution $u_{\tau}$ of \eqref{truneq} satisfying $u_{\tau}\in X^{s, b}(T)\subseteq C([0, T], H^{s})$,  and
		\[
		\|u-u_{\tau}\|_{L^{\infty}([0, T^{}_{}], L^2)}\lesssim \|u-u_{\tau}\|_{X^{0, b}(T_{})}\le C_T\tau^{\frac{s}{2}}.
		\]
	\end{proposition}
	\begin{proof}
		For any $T<T_{s}$, there exists a constant $C_u>0$ such that
		\[
		\|u\|_{X^{s, b}(T)}\le C_u.
		\]
		In the following, we will construct the solution $u_{\tau}$ of \eqref{truneq} on $[0, T]$ and prove the proposition by induction.
		For $\beta$ to be determined (see \eqref{beta} below), we denote by $U^{\beta}_n, U^{\tau,\beta}_{n}\in B(0, R)\subset X^{s, b}$ the fixed-points of $F_{w_1}^{\beta}(v)$ \eqref{functional} and $F_{w_2}^{\tau,\beta}(v)$ \eqref{truncatedfunctional} with $w_1 = u((n-1)\beta)$ and $w_2 = u_{\tau}((n-1)\beta)$, which satisfy $U_n^{\beta}|_{[0, \beta]}=u|_{[(n-1)\beta, n\beta]}$ and $U_{n}^{\tau,\beta}|_{[0, \beta]}=u_{\tau}|_{[(n-1)\beta, n\beta]}$, respectively, where $1\le n\le T/\beta$ and $R=2C(C_u+1)$, with $C>1$ chosen below. Note that the choice of $R$ is independent of $n$ and $u_{\tau}$.
		

		For $k=0$, by using \eqref{lemma16}, we have
		\begin{align}\label{uutau0}
			\|u(0)-u_{\tau}(0)\|_{L^2}\le C_1\|u(0)\|_{H^s}\tau^{\frac{s}{2}}.
		\end{align}
		Next, we denote $N=T/\beta\in\mathbb{N}$ and recall that $\beta$ will be determined later. Suppose that for all $0\le k\le n-1$, the following holds
		\[
		\|U^{\beta}_{k}-U_{k}^{\tau, \beta}\|_{X^{0,b}(\beta)}\le C_{0, k}\tau^{\frac{s}{2}}.
		\]
		It will be shown below that these constants $C_{0,k}$ are uniformly bounded.
		When $k=n$, it holds that
		\begin{align*}
			U^{\beta}_{n}&(t)-U_{n}^{\tau, \beta}(t)\notag\\
			=&\ \eta(t){\rm e}^{it\la \p_x^2\ra}\big(u((n-1)\beta)-u_{\tau}((n-1)\beta)\big)-\frac{i}{2}\la \p_x^2\ra^{-1}\eta(t)\int_0^t {\rm e}^{i(t-\theta)\la \p_x^2\ra}\notag\\
			&\times \bigg[\eta\left(\frac{s}{\beta}\right)\Big(\Pi_{\tau}U^{\beta}_n(\theta)-
			\Pi_{\tau}U^{\tau, \beta}_{n}(\theta)+(1-\Pi_{\tau})U^{\beta}_n(\theta)\notag\\
			&\quad\quad\quad\quad\quad\quad\quad\quad\quad\quad\quad\quad\quad\quad\quad+
			\Pi_{\tau}\ov{U}^{\beta}_n(\theta)-\Pi_{\tau}\ov{U}_{n}^{\tau, \beta}(\theta)+(1-\Pi_{\tau})\ov{U}^{\beta}_n\Big)\bigg]d\theta\notag\\
			&-\frac{i}{4}\la \p_x^2\ra^{-1}\p_x^2\;\eta(t)\int_0^t {\rm e}^{i(t-\theta)\la \p_x^2\ra}\eta\left(\frac{\theta}{\beta}\right)\bigg[\left(1-\Pi_{\tau}\right)\left(U^{\beta}_n(\theta)+\ov{U}^{\beta}_n(\theta)\right)^2\notag\\
			&\quad\quad\quad\quad\quad\quad\quad\,\,\, -\Pi_{\tau}\left[\left(\Pi_{\tau}U_{n}^{\tau, \beta}(\theta)+\Pi_{\tau}\ov{U}_{n}^{\tau, \beta}(\theta)\right)^2-\left(U^{\beta}_n(\theta)+\ov{U}^{\beta}_n(\theta)\right)^2\right]\bigg]d\theta.
		\end{align*}
		Denote $r_j^s=\|u(j\beta)-u_{\tau}(j\beta)\|_{H^s}$ for $0\le j\le N$. By applying \eqref{lemma13}, \eqref{lemma15}, \eqref{lemma14} and \eqref{lemma16}, and noting that
		\be\label{cu}
		\|U^{\beta}_n\|_{X^{0,b}}\le \|U^{\beta}_n\|_{X^{s,b}}\le C_u,
		\ee
		we obtain
		\begin{align*}
			\|U^{\beta}_n-U_{n}^{\tau, \beta}\|_{X^{0,b}}
			\le&\ C_2r^0_{n-1}+C_2\beta^{1-b}\left(\|U_n^{\beta}-U_{n}^{\tau, \beta}\|_{X^{0,b}}+\tau^{\frac{s}{2}}\|U_n^{\beta}\|_{X^{s,b}}\right)\notag\\
			& +C_2\beta^{1-b}\big(\tau^{\frac{s}{2}}\|U_n^{\beta}\|_{X^{s,b}}^2+
			\|U_n^{\beta}-U_{n}^{\tau, \beta}\|_{X^{0,b}}\notag\\
			&\quad\quad\quad\quad\quad\quad\quad\quad\ \quad\quad\quad\times(\|U^{\beta}_n\|_{X^{0,b}}+\|U_n^{\beta}-U_{n}^{\tau,\beta}\|_{X^{0,b}})\big)\notag\\
			\le&\ C_2\left(r^0_{n-1}+\beta^{1-b}C_u\tau^{\frac{s}{2}}+\beta^{1-b}C_u^2\tau^{\frac{s}{2}}\right)\notag\\
			& +C_2\beta^{1-b}\left(1+C_u\right)\|U_n^{\beta}-U_{n}^{\tau, \beta}\|_{X^{0,b}}+C_2\beta^{1-b}\|U_n^{\beta}-U_{n}^{\tau, \beta}\|_{X^{0,b}}^2,
		\end{align*}
		where $C_2$ only depends on $b$ and $\eta$.
		Thus, by choosing an appropriate $\beta\in(0, 1)$ (as will be specified in \eqref{beta} below), and using \eqref{lemma11}, we obtain
		\begin{align}\label{rn0}
			r_n^0\le \widetilde{C} \|U_n^{\beta}-U_{n}^{\tau, \beta}\|_{X^{0,b}}\le C_3\left(r^0_{n-1}+C_4\tau^{\frac{s}{2}}\right),
		\end{align}
		where $C_4=\frac{C_u}{2C_2}$, which implies that
		\begin{align*}
			r^0_{n}+\frac{C_3C_4}{C_3-1}\tau^{\frac{s}{2}}\le C_3\left(r^0_{n-1}+\frac{C_3C_4}{C_3-1}\tau^{\frac{s}{2}}\right),
		\end{align*}
		where we may assume $C_3>1$ without loss of generality. Solving this recursion yields
		\begin{align*}
			r^0_{n}+\frac{C_3C_4}{C_3-1}\tau^{\frac{s}{2}}\le C_3\left(r^0_{n-1}+\frac{C_3C_4}{C_3-1}\tau^{\frac{s}{2}}\right)\le \ldots\le (C_3)^n\left(r_{0}^0+\frac{C_3C_4}{C_3-1}\tau^{\frac{s}{2}}\right),
		\end{align*}
		which implies
		\[
		r^0_{n}\le (C_3)^n\left(C_1\|u(0)\|_{H^s}+\frac{C_3C_4}{C_3-1}\right)\tau^{\frac{s}{2}}-\frac{C_3C_4}{C_3-1}\tau^{\frac{s}{2}}\\
		=:C_{0, n}\tau^{\frac{s}{2}}.
		\]
		
		From this it follows that the sequence $\{C_{0, j}\}_{j=1}^{n}$ has an upper bound
		\begin{align*}
			C_0={\rm e}^{\frac{T}{\beta}\log C_3}\left(C_1\|u_0\|_{H^s}+\frac{C_3C_4}{C_3-1}\right)-\frac{C_3C_4}{C_3-1},
		\end{align*}
		which is independent of $n$. Hence, we obtain
		\begin{align}\label{r0n}
			r^0_n \le C_0\tau^{\frac{s}{2}}.
		\end{align}
		From \eqref{rn0}, we derive that
		\begin{align*}
			\|U_n^{\beta}-U_{n}^{\tau,\beta}\|_{X^{0,b}}\le C_3\left(C_{0, n-1}+C_4\tau^{\frac{s}{2}}\right)\le C_3\left(C_0+C_4\right)\tau^{\frac{s}{2}}.
		\end{align*}
		By using \eqref{lemma17}, we obtain
		\begin{equation*}
			\begin{aligned}
			\|U^{\tau,\beta}_{n}\|_{X^{s,b}}&\le \|\Pi_{\tau}U^{\beta}_{n}\|_{X^{s,b}}+\|\Pi_{\tau}U^{\beta}_{n}-U^{\tau,\beta}_{n}\|_{X^{s,b}}\\
				&\le C_u+\tau^{-\frac{s}{2}}\|U^{\beta}_{n}-U^{\tau, \beta}_{n}\|_{X^{0,b}}\\
				&\le C_u+C_3(C_0+C_4),
			\end{aligned}
		\end{equation*}
		which gives the boundedness of $U_{n}^{\tau,\beta}$ in $X^{s, b}$.
		
		Taking $\tau_0=\left(1/C_0\right)^{2/s}$ and applying \eqref{r0n}, we have for $\tau\le \tau_0$
		\begin{align*}
			\|u_{\tau}(n\beta)\|_{L^2}\le \|u(n\beta)\|_{L^2}+r_n^0\le C_5C_u+1,
		\end{align*}
		where $C_5$ is a constant from the embedding
		$X^{0, b}\hookrightarrow L^{\infty}(\mathbb{R},L^2)$, where $b>\tfrac{1}{2}$.
		Recalling the restrictions imposed by inequalities \eqref{cu} and \eqref{rn0}, we can choose
		\begin{align}\label{beta}
			\beta\le\Big[ \frac{1}{2C\big(1+R\big)}\Big]^{1/(1-b)}<1,
		\end{align}
		such that $T/\beta\in \mathbb{N}$, where $R=2C(C_u+1)$ and $C=\max\{1, C_2, C_5\}$. Using $u_{\tau}(n\beta)$ as the initial value for the next iteration, ultimately we obtain the solution $u_{\tau}\in C([0, (n+1)\beta], L^2)\cap X^{s,b}((n+1)\beta)$. The uniqueness is guaranteed by a similar proof as in \cite{farah2009local} and the proof is completed.
	\end{proof}
	
	\section{Discrete Bourgain spaces and some essential estimates}\label{DBS}
	In this section, we revisit the definition of discrete Bourgain spaces and investigate their properties. This framework was initially introduced by Ostermann, Rousset \& Schratz in their pioneering work \cite{Ostermann2022JEMS}. We also refer to \cite{ji2023Lowregularity,Ostermann2022splitting,rousset2022Convergence} and references therein.
	
	The Fourier transform of a sequence of periodic functions $\{u_n(x)\}_{n\in \mathbb{Z}}$, where $x\in\mathbb{T}$, is defined as
	\begin{align*}
		\mathcal{F}_{\tau, x}(u_n)(\sigma, k)=\widetilde{u_n}(\sigma, k)=\tau\sum\limits_{m\in \mathbb{Z}}\wh{u_m}(k){\rm e}^{-im\tau\sigma}, \ \ \wh{u_m}(k)=\frac{1}{2\pi}\int_{-\pi}^{\pi}u_m(x){\rm e}^{-ikx}dx.
	\end{align*}
	Obviously it satisfies Parseval's identity: $\|\widetilde{u_n}\|_{L^2l^2}=\|u_n\|_{l_{\tau}^2L^2}$, where
	\begin{align*}
		\|\widetilde{u_n}\|^2_{L^2l^2}=\int_{-\frac{\pi}{\tau}}^{\frac{\pi}{\tau}}\sum\limits_{k\in \mathbb{Z}}|\widetilde{u_n}(\sigma, k)|^2d\sigma, \quad \|u_n\|^2_{l_{\tau}^2L^2}=\tau\sum\limits_{m\in \mathbb{Z}}\int_{-\pi}^{\pi}|u_m(x)|^2dx.
	\end{align*}
	For a continuous function $g$, the discrete Bourgain space $X^{s,b}_{\tau, \sigma=g(k)}$ with $s\ge 0$, $b\in \mathbb{R}$, $\tau>0$ is defined by
	\be\label{disdef}
	\|u_n\|_{X^{s,b}_{\tau, \sigma=g(k)}}=\left\|\la k\ra^s\la d_{\tau}(\sigma-g(k))\ra^b\widetilde{u_n}(\sigma, k)\right\|_{L^2l^2},
	\ee
	where $d_{\tau}(\sigma)=\frac{{\rm e}^{i\tau\sigma}-1}{\tau}$ is a $\frac{2\pi}{\tau}$ periodic function satisfying $|d_{\tau}(\sigma)|\sim|\sigma|$ for $|\tau\sigma|\le \pi$.
	Note that, in order to keep in line with the definition of Bourgain spaces in Section~3, this notation is different from that in \cite{ji2023Lowregularity,Ostermann2022splitting,Ostermann2022JEMS, rousset2022Convergence}, where
	\be\label{refer}
	\|u_n\|_{X^{s,b}_{\tau, \sigma=g(k)}}=\left\|\la k\ra^s\la d_{\tau}(\sigma+g(k))\ra^b\widetilde{u_n}(\sigma, k)\right\|_{L^2l^2},\,\,\, \widetilde{u_n}(\sigma, k)=\tau\sum\limits_{m\in \mathbb{Z}}\wh{u_m}(k){\rm e}^{im\tau\sigma}.
	\ee
	One can easily verify that these two definitions yield the same norm and correspondingly the same discrete Bourgain space.
	
	In the following, we always write $\|u_n\|_{X^{s, b}_{\tau, \sigma=\sqrt{k^4+1}}}$ as $\|u_n\|_{X^{s, b}_{\tau}}$ for short.
	For further simplification, we present the following lemma.
	\begin{lemma}\label{equvi}
		For $s\ge 0$, $b\in \mathbb{R}$ and $\tau>0$, it holds that
		\begin{align*}
			\left\|\la k\ra^s\la d_{\tau}(\sigma-\sqrt{k^4+1})\ra^b\widetilde{u_n}(\sigma, k)\right\|_{L^2l^2}\sim \left\|\la k\ra^s\la d_{\tau}(\sigma-k^2)\ra^b\widetilde{u_n}(\sigma, k)\right\|_{L^2l^2}.
		\end{align*}
	\end{lemma}
	\begin{proof}
		It suffices to show that
		\[
		\la d_{\tau}(\sigma-\sqrt{k^4+1})\ra\sim\la d_{\tau}(\sigma-k^2)\ra.
		\]
		Denote $g_{\tau}(x)={\rm e}^{i\tau x}-1$, $x\in \mathbb{R}$. It is easily observed that $|g_{\tau}|$ is an even function satisfying
		\begin{align*}
			|g_{\tau}(x+y)|\le |g_{\tau}(x)|+|g_{\tau}(y)|,\quad
			\big| |g_{\tau}(x)|-|g_{\tau}(y)|\big| \le |g_{\tau}(x-y)|.
		\end{align*}
		Applying the above inequalities and noticing $|g_{\tau}(x)|\le \tau|x|$, one obtains
		\begin{align*}
			\frac{\tau+|g_{\tau}(\sigma-\sqrt{k^4+1})|}{\tau+|g_{\tau}(\sigma-k^2)|}\le&\ \frac{\tau+|g_{\tau}(\sigma-k^2)|+|g_{\tau}(\sqrt{k^4+1}-k^2)|}
			{\tau+|g_{\tau}(\sigma-k^2)|}\\
			\le&\ 1+\frac{\tau}{\tau+|g_{\tau}(\sigma-k^2)|}\le 2,
		\end{align*}
		and
		\begin{align*}
			\frac{\tau+|g_{\tau}(\sigma-\sqrt{k^4+1})|}{\tau+|g_{\tau}(\sigma-k^2)|}\ge&\ \frac{\tau+|g_{\tau}(\sigma-\sqrt{k^4+1})|}{\tau+|g_{\tau}(\sigma-\sqrt{k^4+1})|+
				|g_{\tau}(\sqrt{k^4+1}-k^2)|}\\
			=&\ 1-\frac{|g_{\tau}(\sqrt{k^4+1}-k^2)|}{\tau+|g_{\tau}(\sigma-\sqrt{k^4+1})|
				+|g_{\tau}(\sqrt{k^4+1}-k^2)|}\\
			\ge&\ 1-\frac{|g_{\tau}(\sqrt{k^4+1}-k^2)|}{\tau+|g_{\tau}(\sqrt{k^4+1}-k^2)|}\\
			=&\ \frac{1}{1+|g_{\tau}(\sqrt{k^4+1}-k^2)|/\tau}\ge\frac{1}{2},
		\end{align*}
		which completes the proof.
	\end{proof}

	Lemma \ref{equvi} shows that
	\be\label{conj}
	\|u_n\|_{X^{s,b}_\tau}=\|u_n\|_{X^{s,b}_{\tau, \sigma=\sqrt{k^4+1}}}\sim\|u_n\|_{X^{s,b}_{\tau, \sigma=k^2}}=\|\overline{u}_n\|_{X^{s,b}_{\tau, \sigma=-k^2}},\ee
	where $X^{s, b}_{\tau,\sigma=-k^2}$ is exactly the discrete Bourgain space defined in \cite{Ostermann2022splitting, Ostermann2022JEMS} for the Schr\"odinger equation with norm
	\[\|u_n\|_{X^{s,b}_{\tau, \sigma=-k^2}}=\left\|\la k\ra^s\la d_{\tau}(\sigma+k^2)\ra^b\widetilde{u_n}(\sigma, k)\right\|_{L^2l^2}.\]
	
	By using the equality \eqref{conj} and the properties established for $X^{s,b}_{\tau, \sigma=-k^2}$ in \cite{Ostermann2022splitting, Ostermann2022JEMS}, we can readily derive the following lemma.
	
	\begin{lemma}\label{4geyinli}
		For given $\{u_n(x)\}_{n\in\mathbb{Z}}$ and $\{v_n(x)\}_{n\in\mathbb{Z}}$, $x\in\mathbb{T}$, we have
		\begin{align}
			&\sup\limits_{\delta\in [-4,4]}\|{\rm e}^{i\tau\delta\la \p_x^2\ra}u_n\|_{X^{s,b}_{\tau}}\lesssim \|u_n\|_{X^{s,b}_{\tau}}, \quad s\ge 0, \ b\in\mathbb{R}, \label{lemma21}\\
			&\|u_n\|_{X^{0, b}_{\tau}}\lesssim \frac{1}{\tau^{b-b^{\prime}}}\|u_n\|_{X_\tau^{0, b^{\prime}}}, \quad b\ge b^{\prime}, \ b, b^{\prime}\in\mathbb{R}, \label{lemma22}\\
			&\|u_n\|_{l_{\tau}^{\infty}H^s}\lesssim \|u_n\|_{X^{s,b}_{\tau}}, \quad s\ge 0, \ b>\tfrac{1}{2}, \label{lemma23}\\
			&\|\Pi_{\tau}u_n\|_{l_{\tau}^4 L^4}\lesssim \|u_n\|_{X_{\tau}^{0,\frac{3}{8}}}, \label{lemma24}\\
			&\|\Pi_{\tau}(\Pi_{\tau}u_n\Pi_{\tau}v_n)\|_{X_{\tau}^{s, 0}}\lesssim \|u_n\|_{X_{\tau}^{s,\frac{3}{8}}}\|v_n\|_{X_{\tau}^{s,\frac{3}{8}}}, \quad s\ge 0. \label{lemma25}
		\end{align}
		Note that \eqref{lemma25} remains valid when substituting $\Pi_{\tau}u_n\Pi_{\tau}v_n$ with either $\Pi_{\tau}\overline{u}_n\Pi_{\tau}v_n$ or $\Pi_{\tau}\overline{u}_n\Pi_{\tau}\overline{v}_n$.
	\end{lemma}
	\begin{proof}
		Thanks to Lemma \ref{equvi} and \eqref{conj},
		\eqref{lemma21}-\eqref{lemma24} can be directly derived by referring to \cite[Lemmas 3.1, 3.6 and Remarks 3.2, 3.3]{Ostermann2022JEMS}.
		It remains to show \eqref{lemma25}. Utilizing \eqref{kp1} and \eqref{lemma24} results in
		\begin{align*}
			\|\Pi_{\tau}(\Pi_{\tau}u_n\Pi_{\tau}v_n)&\|_{X_{\tau}^{s, 0}}\sim \left\|\la \p_x\ra^{s}\Pi_{\tau}(\Pi_{\tau}u_n\Pi_{\tau}v_n)\right\|_{X_{\tau}^{0, 0}}\\
			&\lesssim \|\la \p_x\ra^{s}\Pi_{\tau}u_n\|_{l_{\tau}^4 L^4}\|\Pi_{\tau}v_n\|_{l_{\tau}^4 L^4}+\|\la \p_x\ra^{s}\Pi_{\tau}v_n\|_{l_{\tau}^4 L^4}\|\Pi_{\tau}u_n\|_{l_{\tau}^4 L^4}\\
			&\ = \|\Pi_{\tau}\la \p_x\ra^{s}u_n\|_{l_{\tau}^4 L^4}\|\Pi_{\tau}v_n\|_{l_{\tau}^4 L^4}+\|\Pi_{\tau}\la \p_x\ra^{s}v_n\|_{l_{\tau}^4 L^4}\|\Pi_{\tau}u_n\|_{l_{\tau}^4 L^4}\\
			&\lesssim \|u_n\|_{X^{s,\frac{3}{8}}}\|v_n\|_{X_{\tau}^{s,\frac{3}{8}}},
		\end{align*}
		and the proof is completed.
	\end{proof}
	
	\begin{lemma}\label{guanyueta}
		For given $\eta\in C_{c}^{\infty}(\mathbb{R})$, $s\ge 0$, $b\in\mathbb{R}$, $N\in \mathbb{N}$ and $\{u_n(x)\}_{n\in\mathbb{Z}}$, $x\in\mathbb{T}$, we have
		\begin{align}
			&\|\eta(n\tau)u_n\|_{X_{\tau}^{s,b}}\lesssim \|u_n\|_{X_{\tau}^{s,b}}, \label{lemma31}\\
			&\big\|\eta\big(\frac{n\tau}{T}\big)u_n\big\|_{X_{\tau}^{s, b^{\prime}}}\lesssim T^{b-b^{\prime}}\|u_n\|_{X^{s,b}_{\tau}},  \,\,\, -\tfrac{1}{2}<b^{\prime}\le b<\tfrac{1}{2},\ T=N\tau\in(0,1],\label{lemma32}\\
			&\|\eta(n\tau){\rm e}^{in\tau\la \p_x^2\ra}f\|_{X_{\tau}^{s,b}}\lesssim\|f\|_{H^s}, \label{lemma33}\\
			& \bigg\|\eta(n\tau)\tau\sum\limits_{m=0}^{n}{\rm e}^{i(n-m)\tau\la \p_x^2\ra}u_m\bigg\|_{X_{\tau}^{s,b}}\lesssim\|u_n\|_{X^{s,b-1}_{\tau}}, \quad b>\tfrac{1}{2}.\label{lemma34}
		\end{align}
	\end{lemma}
	\begin{proof}
		Employing \eqref{conj} and the corresponding results in \cite[Lemma 4.1]{Ostermann2022splitting},  we arrive at \eqref{lemma31} and\eqref{lemma32} straightforwardly.
		Taking note of Lemma \ref{equvi} and following the proof of \cite[Lemma 3.4]{Ostermann2022JEMS} line by line, we get \eqref{lemma33} and \eqref{lemma34} in a similar way.
	\end{proof}

	Furthermore, we will give an estimate about the relationship between the norms of Bourgain spaces and their discrete versions.
	\begin{lemma}\label{relation}
		For a sequence of functions $\{u_{\tau}(n\tau+\theta,x)\}_{n\in\mathbb{Z}}$ with $\theta\in[0,3\tau]$ and $\tau\in(0,1)$, it holds
		\begin{align*}
			\sup\limits_{\theta\in[0, 3\tau]}\|u_{\tau}(n\tau+\theta,x)\|_{X^{s,\frac{3}{8}}_{\tau}}\lesssim \|u_{\tau}\|_{X^{s,\frac{3}{8}}}+\tau^{\frac{9}{16}}\|u_\tau\|_{X^{s, \frac{15}{16}}},
		\end{align*}
		where $s\ge 0$.
	\end{lemma}
	\begin{proof}
		By taking note of the definition in \eqref{norm} and applying a similar argument as in \cite{ji2023Lowregularity}, we are led to this assertion. For the details, we refer to the proof of \cite[Lemma 4.1]{ji2023Lowregularity}.
	\end{proof}

	\section{Error estimates for the FLTS scheme}\label{LRLT}
	According to Proposition \ref{wellposednessforpj}, the $L^2$ norm of the difference between $u$ and $u_{\tau}$ is bounded by $\mathcal{O}(\tau^{\frac{s}{2}})$ on compact time intervals. Therefore, it suffices to estimate the error $e_n:=u_{\tau}(t_n)-u^n$.
	
	Firstly, the local error can be expressed as
	\[
	u_{\tau}(t_{n+1})-\varPhi^{\tau}(u_{\tau}(t_n))=-i\la \p_x^2\ra^{-1}e^{i\tau\la \p_x^2\ra}\bigg(\mathcal{E}^1_n+\mathcal{E}^2_n+\p_x^2 \sum\limits_{j=3}^5\mathcal{E}^j_n\bigg),
	\]
	where $\varPhi^{\tau}$ is given by \eqref{numflow} and
	\begin{align*}
		\mathcal{E}^1_n&=\int_0^{\tau}({\rm e}^{-i\theta\la \p_x^2\ra}-1)\Pi_{\tau}\Re u_{\tau}(t_n+\theta)d\theta,\notag\\
		\mathcal{E}^2_n&=\int_0^{\tau}\Pi_{\tau} \Re\big(u_{\tau}(t_n+\theta)-u_{\tau}(t_n)\big)d\theta,\notag\\
		\mathcal{E}^3_n&=\int_0^{\tau}({\rm e}^{-i\theta\la \p_x^2\ra}-1)\Pi_{\tau}[\Pi_{\tau}\Re u_{\tau}(t_n+\theta)]^2d\theta,\notag\\
		\mathcal{E}^4_n&=\int_0^{\tau}\Pi_{\tau}[\Pi_{\tau}\Re u_{\tau}(t_n+\theta)]^2-\Pi_{\tau}\big[\Pi_{\tau}\Re u_{\tau}(t_n+\theta)\Pi_{\tau}\Re u_{\tau}(t_n)\big]d\theta,\notag\\
		\mathcal{E}^5_n&=\int_0^{\tau}\Pi_{\tau}\big[\Pi_{\tau}\Re u_{\tau}(t_n+\theta)\Pi_{\tau}\Re u_{\tau}(t_n)\big]-\Pi_{\tau}[\Pi_{\tau}\Re u_{\tau}(t_n)]^2d\theta.\notag
	\end{align*}
	Below, we will establish local error estimates and ultimately present the global error for \eqref{numflow}.
	\subsection{Local error estimates}
	Applying Duhamel's formula for \eqref{truneq}, we obtain
	\be\label{onesgap}
	\begin{split}
		u_{\tau}(t_{n}+\theta)-u_{\tau}(t_n)=&\ ({\rm e}^{i\theta\la \p_x^2\ra}-1)u_{\tau}(t_n)\\
		&-i\la \p_x^2\ra^{-1}\int_{0}^{\theta}{\rm e}^{i(\theta-\xi)\la \p_x^2\ra}\Pi_{\tau}\Re u_{\tau}(t_n+\xi)d\xi\\
		&-i\la \p_x^2\ra^{-1}\p_x^2\int_{0}^{\theta}{\rm e}^{i(\theta-\xi)\la \p_x^2\ra}\Pi_{\tau}[\Pi_{\tau}\Re u_{\tau}(t_n+\xi)]^2d\xi.
	\end{split}
	\ee
	Thanks to Proposition \ref{wellposednessforpj}, for $s\in(0,2]$, we have $u_{\tau}\in X^{s,\frac{15}{16}}(T)$. According to the definition of localized Bourgain spaces (cf.~\eqref{locali}), we will hereafter denote by $u_{\tau}$ any extension of the solution to \eqref{truneq} from $[0, T]$ to $\mathbb{R}$ which satisfies
	\begin{align}\label{boundedness}
		\|u_{\tau}\|_{X^{s,\frac{15}{16}}}\le 2\|u_{\tau}\|_{X^{s,\frac{15}{16}}(T)}\lesssim 1.
	\end{align}
	By utilizing Lemma \ref{relation}, \eqref{boundedness} and the fact that $\tau$ is bounded, we observe that
	\begin{align}\label{disbound}
		\sup\limits_{\theta\in [0, \tau]}\|u_{\tau}(t_n+\theta)\|_{X_{\tau}^{s, \frac{3}{8}}}\lesssim 1,
	\end{align}
	which allows us to establish the following estimate for the local error.

	\begin{lemma}
		For $s\in(0, 2]$ and $T<T_{s}$, it holds
		\begin{align}\label{localerror}
			\|	u_{\tau}(t_{n+1})-\varPhi^{\tau}(u_{\tau}(t_n))\|_{X^{0,0}_{\tau}}\le C_{T}\tau^{1+\frac{s}{2}}.
		\end{align}
	\end{lemma}
	\begin{proof}
		It suffices to show
		\begin{align*}
			\|\mathcal{E}^i_n\|_{X^{0, 0}_{\tau}}\le C_{T}\tau^{1+\frac{s}{2}}, \quad \quad 1\le i\le 5.
		\end{align*}
		To estimate $\mathcal{E}^i$, it is important to note that the operator $\Pi_{\tau}$ projects onto frequencies less than $\tau^{-\frac{1}{2}}$, and that ${\rm e}^{-i\theta\la \p_x^2\ra}-1$ is a Fourier multiplier in space. Therefore, we observe that for any $\{f(t_n)\}_n\in X^{s,b}_\tau$, the following holds
		\begin{align}\label{etheta-1}
			\sup\limits_{\theta\in[0, \tau]} \|({\rm e}^{-i\theta\la \p_x^2\ra}-1)\Pi_{\tau}f(t_n)\|_{X^{0,b}_{\tau}}\lesssim \tau^{\frac{s}{2}}\|f(t_n)\|_{X^{s,b}_{\tau}},\quad b\in\mathbb{R},
		\end{align}
		which together with \eqref{disbound} yields
		\begin{align*}
			\|\mathcal{E}^1_n\|_{X^{0, 0}_{\tau}}&\lesssim \tau \sup\limits_{\theta\in[0, \tau]} \|({\rm e}^{-i\theta\la \p_x^2\ra}-1)\Pi_{\tau}\Re u_{\tau}(t_n+\theta)\|_{X^{0,0}_{\tau}}\notag\\
			&\lesssim \tau^{1+\frac{s}{2}} \sup\limits_{\theta\in[0, \tau]}\|u_{\tau}(t_n+\theta)\|_{X^{s,0}_{\tau}}\le C_{T}\tau^{1+\frac{s}{2}},
		\end{align*}
		for $s\in (0, 2]$.
		Similarly, applying \eqref{etheta-1}, \eqref{lemma25} and \eqref{disbound}, we arrive at
		\begin{align*}
			\|\mathcal{E}^3_n\|_{X^{0, 0}_{\tau}}&\lesssim \tau \sup\limits_{\theta\in[0, \tau]}\|({\rm e}^{-i\theta\la \p_x^2\ra}-1)\Pi_{\tau}[\Pi_{\tau}\Re u_{\tau}(t_n+\theta)]^2\|_{X^{0, 0}_{\tau}}\notag\\
			&\lesssim \tau^{1+\frac{s}{2}}\sup\limits_{\theta\in[0, \tau]}\|\Pi_{\tau}[\Pi_{\tau}u_{\tau}(t_n+\theta)]^2\|_{X^{s, 0}_{\tau}}\notag\\
			&\lesssim \tau^{1+\frac{s}{2}}\sup\limits_{\theta\in[0, \tau]}\|u_{\tau}(t_n+\theta)\|^2_{X^{s, \frac{3}{8}}_{\tau}}\le C_{T}\tau^{1+\frac{s}{2}}.
		\end{align*}
		Plugging \eqref{onesgap} into $\mathcal{E}^4_n$, we can write $\mathcal{E}^4_n=\mathcal{E}^{4,1}_n+\mathcal{E}^{4,2}_n+
		\mathcal{E}^{4,3}_n$, where
		\begin{align*}
			\mathcal{E}^{4,1}_n&=\int_0^{\tau}\Pi_{\tau}\big[\Pi_{\tau}\Re\big(u_{\tau}(t_n+\theta)\big)\Pi_{\tau}\Re\big(({\rm e}^{i\theta\la \p_x^2\ra}-1)u_{\tau}(t_n)\big)\big]d\theta,\\
			\mathcal{E}^{4,2}_n&=\int_0^{\tau}\Pi_{\tau}\Big[\Pi_{\tau}\Re\big(u_{\tau}(t_n+\theta)\big)\\
			&\quad\quad   \ \ \ \ \ \ \ \ \   \ \ \ \times\Pi_{\tau}\Im\Big(\la \p_x^2\ra^{-1}\int_0^\theta{\rm e}^{i(\theta-\xi)\la \p_x^2\ra}\Pi_{\tau}\Re u_{\tau}(t_n+\xi)d\xi\Big)\Big]d\theta,\\
			\mathcal{E}^{4,3}_n&=\int_0^{\tau}\Pi_{\tau}\Big[\Pi_{\tau}\Re\big(u_{\tau}(t_n+\theta)\big)\\
			&\quad \ \  \ \    \  \ \times\Pi_{\tau}{\rm Im}\Big(\la \p_x^2\ra^{-1}\p_x^2\int_0^\theta{\rm e}^{i(\theta-\xi)\la \p_x^2\ra}\Pi_\tau\big[\Pi_\tau \Re u_{\tau}(t_n+\xi)\big]^2d\xi\Big)\Big]d\theta.
		\end{align*}
		Applying \eqref{lemma25}, \eqref{disbound} and \eqref{etheta-1}, one easily gets
		\begin{align*}
			\|\mathcal{E}^{4,1}_n\|_{X_{\tau}^{0, 0}}&\lesssim \tau \sup\limits_{\theta\in[0, \tau]}\left\|\Pi_{\tau}\big[\Pi_{\tau}u_{\tau}(t_n+\theta)\Pi_{\tau}({\rm e}^{i\theta\la \p_x^2\ra}-1)u_{\tau}(t_n)\big]\right\|_{X_{\tau}^{0, 0}}\\
			&\lesssim  \tau^{1+\frac{s}{2}}\sup\limits_{\theta\in[0, \tau]}\|u_{\tau}(t_n+\theta)\|_{X_{\tau}^{0, \frac{3}{8}}}\|u_{\tau}(t_n)\|_{X_{\tau}^{s, \frac{3}{8}}}\le C_{T}\tau^{1+\frac{s}{2}}.
		\end{align*}
		Similarly, utilizing \eqref{lemma25}, \eqref{etheta-1} \eqref{lemma21}, and \eqref{disbound}, we arrive at
		\begin{align*}
			\|\mathcal{E}^{4,2}_n\|_{X_{\tau}^{0, 0}}&\lesssim \tau^{}\sup\limits_{\theta\in[0, \tau]}\Big\|\Pi_{\tau}\Big[\Pi_{\tau}u_{\tau}(t_n+\theta)\\
			&\ \ \  \ \  \ \ \ \ \ \ \ \ \   \ \ \  \ \ \ \ \ \ \  \ \times \Pi_{\tau}\la \p_x^2\ra^{-1}\int_0^\theta{\rm e}^{i(\theta-\xi)\la \p_x^2\ra}u_{\tau}(t_n+\xi)d\xi\Big]\Big\|_{X_{\tau}^{0,0}}\\
			&\lesssim \tau^{2}\sup\limits_{\theta\in[0, \tau]} \Big(\|u_{\tau}(t_n+\theta)\|_{X_{\tau}^{0, \frac{3}{8}}}\sup\limits_{\xi\in [0, \theta]} \|{\rm e}^{i(\theta-\xi)\la \p_x^2\ra}u_{\tau}(t_n+\xi)\|_{X_{\tau}^{0, \frac{3}{8}}}\Big)\\
			&\lesssim \tau^2  \Big(\sup\limits_{\theta\in[0, \tau]} \|u_{\tau}(t_n+\theta)\|_{X_{\tau}^{0, \frac{3}{8}}}\Big)^2\le C_{T}\tau^{2}.
		\end{align*}
		
		Prior to estimating the term $\mathcal{E}^{4,3}_n$, we first estimate the integral that appears within $\mathcal{E}^{4,3}_n$. For $\theta\le\tau$, applying \eqref{lemma22}, \eqref{lemma25} and \eqref{disbound}, we get
		\begin{align*}
			\Big\|\int_0^\theta{\rm e}^{i(\theta-\xi)\la \p_x^2\ra}\Pi_{\tau}[\Pi_{\tau}u_{\tau}(t_n+\xi)]^2d\xi\Big\|_{X_{\tau}^{0, \frac{3}{8}}}
			&\lesssim \theta \sup\limits_{\xi\in [0, \theta]}\Big\|\Pi_{\tau}[\Pi_{\tau}u_{\tau}(t_n+\xi)]^2\Big\|_{X_{\tau}^{0, \frac{3}{8}}}\\
			&\lesssim \frac{\theta}{\tau^{\frac{3}{8}}}\sup\limits_{\xi\in [0, \theta]} \Big\|\Pi_{\tau}[\Pi_{\tau}u_{\tau}(t_n+\xi)]^2\Big\|_{X_{\tau}^{0, 0}}\\
			&\lesssim \theta^{\frac{5}{8}}\sup\limits_{\xi\in [0, \theta]}\|u_{\tau}(t_n+\xi)\|^2_{X_{\tau}^{0, \frac{3}{8}}}\le C_{T}\theta^{\frac{5}{8}}.
		\end{align*}
		Thus, $\mathcal{E}^{4,3}_n$ can be bounded as
		\begin{align*}
			\|\mathcal{E}^{4,3}_n\|_{X_{\tau}^{0, 0}}&\lesssim \tau^{}\sup\limits_{\theta\in[0, \tau]}\Big\|\Pi_{\tau}\Big[\Pi_{\tau}u_{\tau}(t_n+\theta)\\
			&\ \ \  \ \  \ \ \ \ \ \ \ \ \   \ \  \  \times\Pi_{\tau} \la \p_x^2\ra^{-1}\p_x^2\int_0^\theta{\rm e}^{i(\theta-\xi)\la \p_x^2\ra}\big[\Pi_{\tau}u_{\tau}(t_n+\xi)\big]^2d\xi\Big]\Big\|_{X_{\tau}^{0,0}}\notag\\
			&\lesssim \tau^{}\sup\limits_{\theta\in[0, \tau]}\Big(\|u_{\tau}(t_n+\theta)\|_{X_{\tau}^{0, \frac{3}{8}}}  \Big\|\int_0^\theta{\rm e}^{i(\theta-\xi)\la \p_x^2\ra}\Pi_{\tau}(\Pi_{\tau}u_{\tau}(t_n+\xi))^2d\xi\Big\|_{X_{\tau}^{0, \frac{3}{8}}}\Big)\\
			&\lesssim \tau^{\frac{13}{8}}\sup\limits_{\theta\in [0, \tau]}\|u_{\tau}(t_n+\theta)\|^3_{X_{\tau}^{0, \frac{3}{8}}}\\
			&\le C_{T}\tau^{\frac{13}{8}} \le C_{T}\tau^{1+\frac{s}{2}},\quad \mathrm{when}\quad s\le 5/4.
		\end{align*}
		In the case of $s>5/4$, we estimate $\|\mathcal{E}^{4,3}_n\|_{X_{\tau}^{0, 0}}$ in a different way. For $b\in(1/2, 1)$, by using Parseval's identity, H{\" o}lder's inequality, Sobolev embedding, \eqref{lemma11} and \eqref{boundedness}, we obtain
		\begin{align*}
			\|\mathcal{E}^{4,3}_n\|_{X_{\tau}^{0, 0}}&=\|\mathcal{E}^{4,3}_n\|_{l_{\tau}^{2}L^{2}}\\
			&\lesssim \tau^{}\sup\limits_{\theta\in[0, \tau]}\Big(\Big\|\int_0^\theta{\rm e}^{i(\theta-\xi)\la \p_x^2\ra}[\Pi_{\tau}u_{\tau}(t_n+\xi)]^2d\xi\Big\|_{l_{\tau}^{\infty}L^2}
			\|\Pi_{\tau}u_{\tau}(t_n+\theta)\|_{l_{\tau}^{2}L^\infty}\Big)\\
			&\lesssim \tau^{}\sup\limits_{\theta\in[0, \tau]}\Big(\Big\|\int_0^\theta{\rm e}^{i(\theta-\xi)\la \p_x^2\ra}[\Pi_{\tau}u_{\tau}(t_n+\xi)]^2d\xi\Big\|_{l_{\tau}^{\infty}L^2}
			\|\Pi_{\tau}u_{\tau}(t_n+\theta)\|_{l_{\tau}^{2}H^s}\Big)\\
			&\lesssim \tau^{2}\sup\limits_{\theta\in[0, \tau]}\|u_{\tau}(t_n+\theta)\|_{l_{\tau}^{\infty}H^s}^2\sup\limits_{\theta\in[0, \tau]}\|u_{\tau}(t_n+\theta)\|_{l_{\tau}^{2}H^{s}}\\
			&\lesssim \tau^{2}\sup\limits_{\theta\in[0, \tau]}\|u_{\tau}(t_n+\theta)\|_{L^{\infty}H^{s}}^2\sup\limits_{\theta\in[0, \tau]}\|u_{\tau}(t_n+\theta)\|_{l_{\tau}^{2}H^{s}}\\
			&\lesssim \tau^{2}\sup\limits_{\theta\in[0, \tau]}\|u_{\tau}(t_n+\theta)\|_{X_{}^{s, b}}^2\sup\limits_{\theta\in[0, \tau]}\|u_{\tau}(t_n+\theta)\|_{X_{\tau}^{s,0}}\lesssim C_{T}\tau^2.
		\end{align*}

		Combining the estimates of $\mathcal{E}^{4,j}_n$ for $j=1,2,3$, one obtains
		\[
		\|\mathcal{E}^4_n\|_{X^{0, 0}_{\tau}}\le \|\mathcal{E}^{4,1}_n\|_{X^{0, 0}_{\tau}}+\|\mathcal{E}^{4,2}_n\|_{X^{0, 0}_{\tau}}+\|\mathcal{E}^{4,3}_n\|_{X^{0, 0}_{\tau}}\le C_{T}\tau^{1+\frac{s}{2}},\quad s\in (0,2].
		\]
		In a similar way, we obtain that the terms $\mathcal{E}^2_n$ and $\mathcal{E}^5_n$ are bounded by $C_{T}\tau^{1+\frac{s}{2}}$.
		This concludes the proof.
	\end{proof}

	\subsection{Global error estimates}\label{global}
	Now we are in a position to establish the global error estimate.
	\begin{proof}[Proof of Theorem \ref{th1}]
		Taking \eqref{znscheme} and Proposition \ref{wellposednessforpj} into account, we obtain
		\begin{align*}
			\|z(t_n)-z^n\|_{L^2}+\|z_t(t_n)-z_t^n\|_{H^{-2}}&\lesssim \|u(t_n)-u^n\|_{L^2}\\
			&\le \|u(t_n)-u_{\tau}(t_n)\|_{L^2}+\|u_{\tau}(t_n)-u^n\|_{L^2}\\
			&\le C_T\tau^{\frac{s}{2}}+\|e_n\|_{L^2}.
		\end{align*}
		It remains to estimate $\|e_n\|_{L^2}$.
		By definition, the error $e_n$ can be written as follows
		\be\label{globalerror}
		\begin{split}
			e_n &=u_{\tau}(t_n)-u^n=\varPhi^{\tau}(u_{\tau}(t_{n-1}))-		\varPhi^{\tau}(u^{n-1})+u_\tau(t_n)-\varPhi^{\tau}(u_{\tau}(t_{n-1}))\\
			&={\rm e}^{i\tau\la \p_x^2\ra}(u_{\tau}(t_{n-1})-u^{n-1})-i\tau\la \p_x^2\ra^{-1} {\rm e}^{i\tau\la \p_x^2\ra}\Pi_{\tau}\Re\big(u_{\tau}(t_{n-1})-u^{n-1}\big)\\
			&\quad -i\tau\la \p_x^2\ra^{-1}\p_x^2 {\rm e}^{i\tau\la  \p_x^2\ra}\Big[\Pi_{\tau}\big(\Pi_{\tau}\Re u_{\tau}(t_{n-1})\big)^2-\Pi_{\tau}\big(\Pi_{\tau}\Re u^{n-1}\big)^2\Big]\\
			&\quad +u_{\tau}(t_n)-\varPhi^{\tau}(u_{\tau}(t_{n-1}))\\
			&=-i\tau\la \p_x^2\ra^{-1}\sum\limits_{k=0}^{n-1} {\rm e}^{i(n-k)\tau\la \p_x^2\ra}\Pi_{\tau}\Re\big(u_{\tau}(t_k)-u^k\big)\\
			&\quad  -i\tau\la \p_x^2\ra^{-1}\p_x^2\sum\limits_{k=0}^{n-1} {\rm e}^{i(n-k)\tau\la  \p_x^2\ra}\Big[\Pi_{\tau}\big(\Pi_{\tau}\Re u_{\tau}(t_k)\big)^2-\Pi_{\tau}\big(\Pi_{\tau}\Re u^k\big)^2\Big]\\
			&\quad +\sum\limits_{k=0}^{n-1} {\rm e}^{i(n-k-1)\tau\la  \p_x^2\ra}\big(u_{\tau}(t_{k+1})-\varPhi^{\tau}(u_{\tau}(
t_k))\big).
		\end{split}
		\ee
		Suppose $\eta\in C^\infty_c:\mathbb{R}\rightarrow [0,1]$ is a smooth function taking $1$ on the interval $[-1, 1]$ and supported in $[-2, 2]$. In the following discussions, we will refer to the truncated version of \eqref{globalerror} as $e_n$, which reads as follows
		\be\label{newen}
		\begin{split}
			e_n&=-i\tau\la \p_x^2\ra^{-1}\eta(t_{n})\sum\limits_{k=0}^{n-1} {\rm e}^{i(n-k)\tau\la \p_x^2\ra} \eta\left(\frac{t_k}{\widetilde{T}}\right)\Pi_{\tau}\Re\big(u_{\tau}(t_k)-
			u^k\big)\\
			&\quad -i\tau\la \p_x^2\ra^{-1}\p_x^2 \ \eta(t_{n})\sum\limits_{k=0}^{n-1} {\rm e}^{i(n-k)\tau\la  \p_x^2\ra} \eta\left(\frac{t_k}{\widetilde{T}}\right)\Big[\Pi_{\tau}\big(\Pi_{\tau}\Re u_{\tau}(t_k)\big)^2\\
			&\quad\qquad\qquad\qquad\qquad\qquad\qquad\qquad -\Pi_{\tau}\big(\Pi_{\tau}\Re u^k\big)^2\Big]+\mathcal{R}_n(\tau, u_{\tau})\\
			&=\mathcal{S}_n(\tau, u_{\tau}, u^n)+\mathcal{R}_n(\tau, u_{\tau}),
		\end{split}
		\ee
		where $\widetilde{T}<\min(1,T)$ is a positive constant to be determined, and $\mathcal{R}_n(\tau, u_{\tau})$ is given by
		\begin{align*}
			\mathcal{R}_n(\tau, u_{\tau})=\eta(t_{n})\sum\limits_{k=0}^{n-1} {\rm e}^{i(n-k-1)\tau\la  \p_x^2\ra}\eta\left(\frac{t_{k+1}}{\widetilde{T}}\right)\big(u_{\tau}(t_{k+1})-\varPhi^{\tau}(u_{\tau}
(t_{k}))\big).
		\end{align*}
		Note that for $0 \le n \le \left[\frac{\widetilde{T}}{\tau}\right]$, we have $e_n = u_{\tau}(t_n) -u^n$.
		
		For $b\in(1/2, 1)$, applying \eqref{lemma34} and \eqref{lemma32}, we obtain
		\begin{align*}
			\|\mathcal{S}_n(\tau, u_{\tau}, u^n)\|_{X^{0, b}_{\tau}}
			\le&\ C_T\widetilde{T}^{1-b}\big(\|\Pi_{\tau}\Re(u_{\tau}(t_n)-u^n)\|_{X^{0, 0}_{\tau}}\\
			&\ +\big\|\Pi_{\tau}\big[\big(\Pi_{\tau}\Re u_{\tau}(t_n)\big)^2-(\Pi_{\tau}\Re u^n\big)^2\big]\big\|_{X^{0, 0}_{\tau}}\big).
		\end{align*}
		The first term on the right-hand side can be bounded as follows:
		\begin{align*}
			&\|\Pi_{\tau}\Re (u_{\tau}(t_n)-u^n)\|_{X^{0, 0}_{\tau}}\lesssim \|e_n\|_{X^{0, 0}_{\tau}}\lesssim \|e_n\|_{X^{0, \frac{3}{8}}_{\tau}}\lesssim \|e_n\|_{X^{0, b}_{\tau}}.
		\end{align*}
		On the other hand, applying \eqref{lemma25} and \eqref{disbound}, one gets
		\begin{align*}
			\big\|\Pi_{\tau}&\big[\big(\Pi_{\tau}\Re u_{\tau}(t_n)\big)^2-(\Pi_{\tau}\Re u^n\big)^2\big]\big\|_{X^{0, 0}_{\tau}}\\
			&\lesssim \|\Pi_{\tau}[(\Pi_{\tau}\Re u_{\tau}(t_n))(\Pi_{\tau} \Re e_n)]\|_{X^{0, 0}_{\tau}}+\|\Pi_{\tau}[(\Pi_{\tau}\Re e_n)(\Pi_{\tau}\Re u^n)]\|_{X^{0, 0}_{\tau}}\\
			& \lesssim \|u_{\tau}(t_n)\|_{X^{0, \frac{3}{8}}_{\tau}}\|e_n\|_{X^{0, \frac{3}{8}}_{\tau}}+\|e_n\|_{X^{0, \frac{3}{8}}_{\tau}}\|u^n\|_{X^{0, \frac{3}{8}}_{\tau}}\\
			&\lesssim \|u_{\tau}(t_n)\|_{X^{0, \frac{3}{8}}_{\tau}}\|e_n\|_{X^{0, \frac{3}{8}}_{\tau}}+\|e_n\|_{X^{0, \frac{3}{8}}_{\tau}}\big(\|u_{\tau}(t_n)\|_{X^{0, \frac{3}{8}}_{\tau}}+\|e_n\|_{X^{0, \frac{3}{8}}_{\tau}}\big)\\
			&\lesssim \|e_n\|_{X^{0, b}_{\tau}}+\|e_n\|^2_{X^{0, b}_{\tau}},
		\end{align*}
		which yields
		\be\label{sta1}
		\|\mathcal{S}_n(\tau, u_{\tau}, u^n)\|_{X^{0, b}_{\tau}}\le C_T \tilde{T}^{1-b}\big(\|e_n\|_{X^{0, b}_{\tau}}+\|e_n\|^2_{X^{0, b}_{\tau}}\big).\ee
		This together with \eqref{newen}, \eqref{lemma23}, \eqref{lemma34}, \eqref{lemma32} and the local error estimate \eqref{localerror} yields for $b\in (1/2, 1)$
		\begin{align*}
			\|e_n\|_{l_{\tau}^{\infty}L^2}&\lesssim\|e_n\|_{X^{0, b}_{\tau}}\notag\\
			&\le\|\mathcal{R}_n(\tau, u_{\tau})\|_{X^{0, b}_{\tau}}+\|\mathcal{S}_n(\tau, u_{\tau}, u^n)\|_{X^{0, b}_{\tau}}\notag\\
			&\lesssim \tau^{-1}\|u_{\tau}(t_n)-\varPhi^{\tau}(u_{\tau}(t_{n-1}))\|_{X^{0, 0}_{\tau}}+\|\mathcal{S}_n(\tau, u_{\tau}, u^n)\|_{X^{0, b}_{\tau}}\notag\\
			&\le C_T\tau^{\frac{s}{2}}+C_T\widetilde{T}^{1-b}\left[\|e_n\|_{X^{0, b}_{\tau}}+\|e_n\|^2_{X^{0, b}_{\tau}}\right].
		\end{align*}
		By selecting $\widetilde{T}$ small enough, we conclude that
		\begin{align*}
			\|e_n\|_{X^{0, b}_{\tau}}\le C_T\tau^{\frac{s}{2}},
		\end{align*}
		which proves the desired estimate for $0\le n\tau \le \widetilde{T}$. Since the selection of $\widetilde{T}$ depends only on $T$, we can repeat the argument for $\widetilde{T}\le n\tau \le 2\widetilde{T}$, $2\widetilde{T}\le n\tau \le 3\widetilde{T}$, $\ldots$, to obtain the estimate for $0\le n\tau \le T$.
	\end{proof}
	
	\section{Numerical experiments}\label{num}
	In this section, we conduct several numerical tests using the newly proposed FLTS scheme \eqref{numflow} to demonstrate the agreement of our computational results with the theoretical predictions on convergence.The spatial discretization is implemented through the Fourier pseudospectral method, enabling us to efficiently compute each step using the Fast Fourier Transform (FFT) in $\mathcal{O}(M\log M)$ operations, where $M$ represents the number of spatial grid points.  We measure the error at time $t_n$ using the $L^2$ norm $\|z(t_n)-z^n\|_{L^2}+\|\partial_t z(t_n)-z_t^n\|_{H^{-2}}$.
	
	\begin{figure}[h!]
		\center
		\hspace{-27pt}
		\subfigure{
			\begin{minipage}[c]{0.4\linewidth}
				\centering
				\includegraphics[height=5.2cm, width=1.3\linewidth]{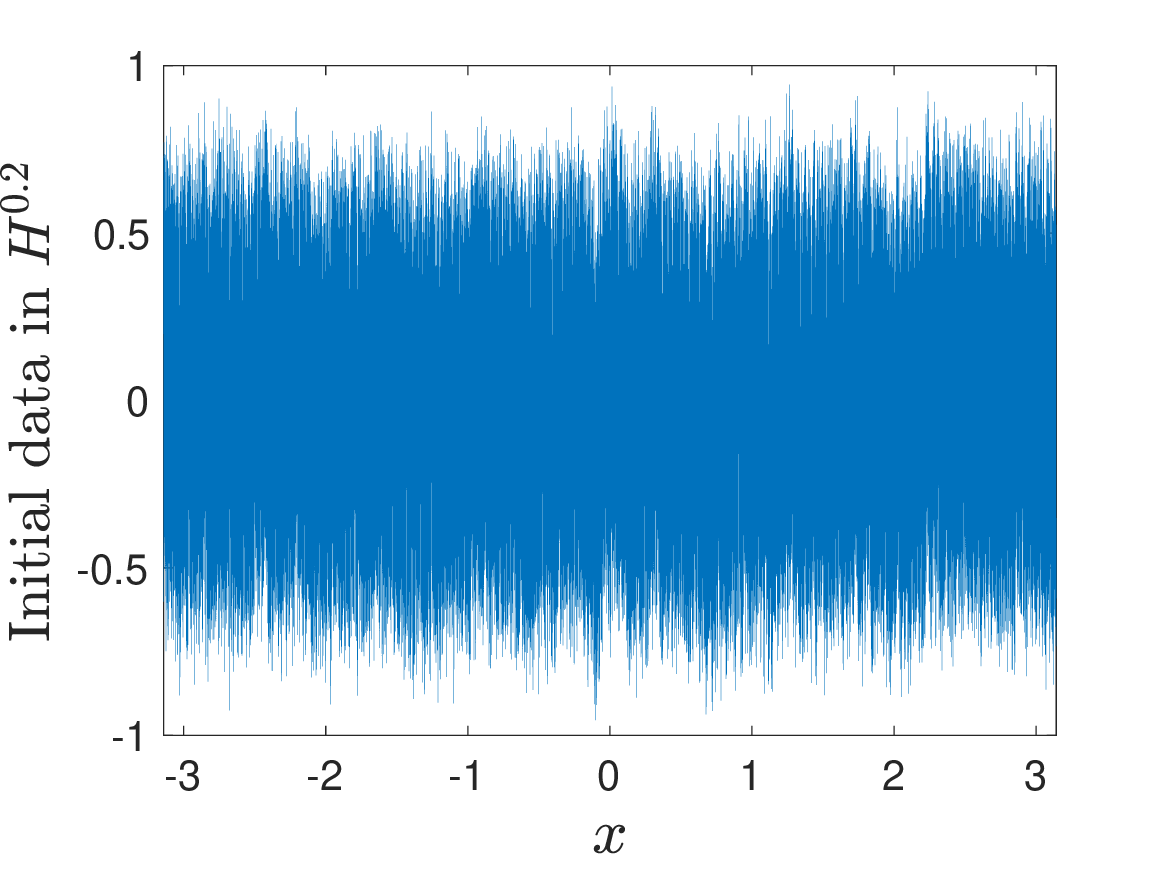}
			\end{minipage}
		}\hspace{26pt}
		\subfigure{
			\begin{minipage}[c]{0.4\linewidth}
				\centering
				\includegraphics[height=5.2cm, width=1.3\linewidth]{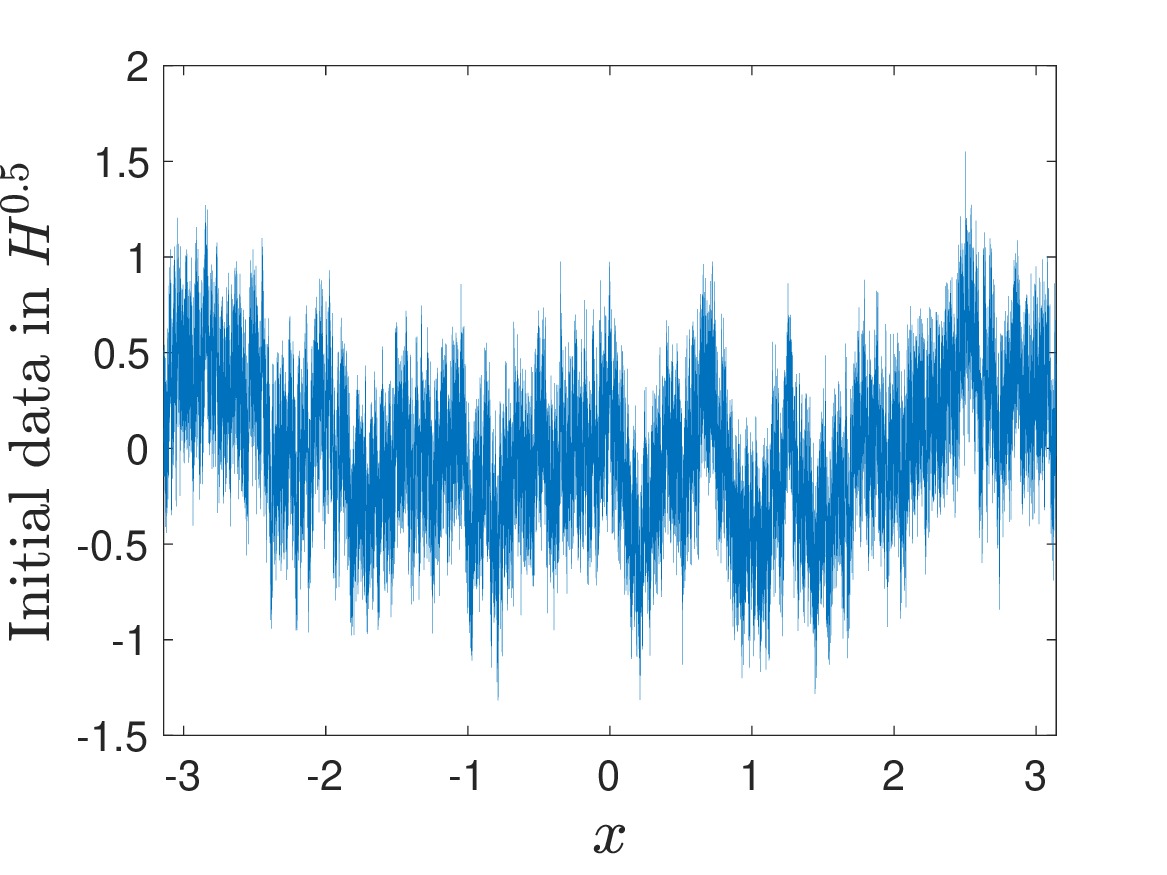}
			\end{minipage}
		}
		\vspace{-4mm}
		\caption{Left: initial value $Z_0$ $\in$ $H^{0.2}$. Right: initial value $Z_0$ $\in$ $H^{0.5}$.} \label{initia}
	\end{figure}
	
	To generate the initial data with the desired regularity, we choose the spatial mesh size $\Delta x=2\pi /M$ with $M=2^{17}$ and the grid points are denoted by $x_j=-\pi+j\Delta x$, $0\leq j<M$. Additionally, we select a random vector uniformly distributed between $0$ and $1$, expressed as rand($1,M$) in Matlab, and denote it as $Z=(z_0, \dots, z_{M-1})$. The operator $|\partial_{x, M}|^{-s}$ acts as an inverse derivative, transforming a function from $L^2(\mathbb{T})$ to the Sobolev space $H^{s}(\mathbb{T})$ by the formula:
	\[
	|\partial_{x, M}|^{-s}f=\sum\limits_{k=-M/2, k\neq 0}^{M/2-1}| k|^{-s} \widehat{f}_k {\rm e}^{ikx},
	\]
	where $s$ is a real number. To normalize the initial data, we begin by computing
	\[
	Z_1=|\partial_{x, M}|^{-s} Z,
	\]
	and then define $Z_0$ as
	\[
	Z_0=\frac{Z_1+c\|Z_1\|_{\infty}}{\|Z_1+c\|Z_1\|_{\infty}\|_{H^{s}}},
	\]
	where $c$ is a randomly chosen number using the command rand($1$). The resulting $Z_0$ is a normalized function in $H^{s}(\mathbb{T})$. Figure \ref{initia} depicts the initial data obtained through the described procedure for  $s=0.2$ and $s=0.5$, respectively.

	Figures \ref{1-1} and \ref{1-2} illustrate the temporal errors of the scheme \eqref{numflow} with initial data in $H^{0.2}$, $H^{1/3}$, $H^{0.5}$ and $H^{0.8}$, respectively. The reference solution is obtained using the numerical solution from the FLTS method with an extremely small time step size $\tau=10^{-7}$. From Figures \ref{1-1} and \ref{1-2}, it can be seen that the results of our numerical experiments agree well with our theoretical analysis (cf. Theorem \ref{th1}), demonstrating that the scheme converges at $O(\tau^{s/2})$ for initial value in $H^s$.
	
	\begin{figure}[htbp]
		\center
		\hspace{-27pt}
		\subfigure{
			\begin{minipage}[c]{0.4\linewidth}
				\centering
				\includegraphics[height=5cm, width=1.3\linewidth]{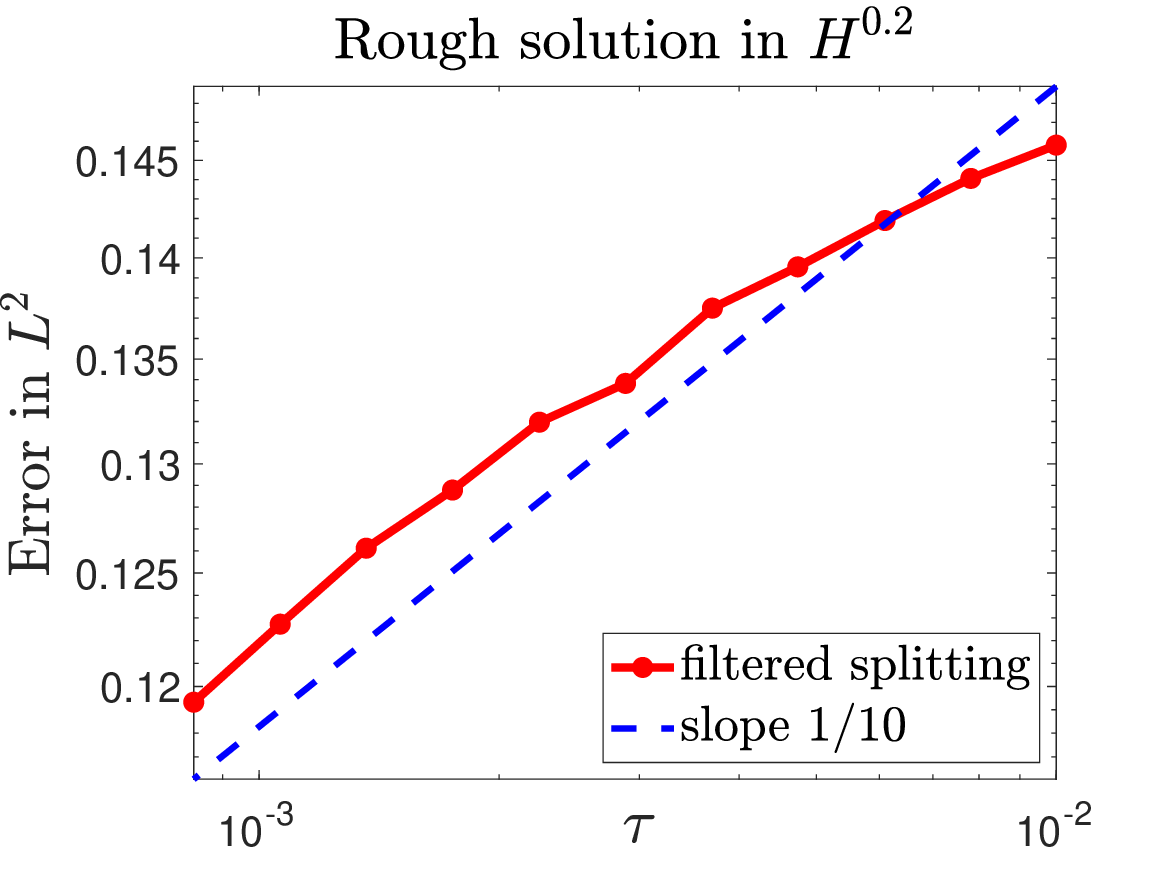}
			\end{minipage}
		}\hspace{26pt}
		\subfigure{
			\begin{minipage}[c]{0.4\linewidth}
				\centering
				\includegraphics[height=5cm, width=1.3\linewidth]{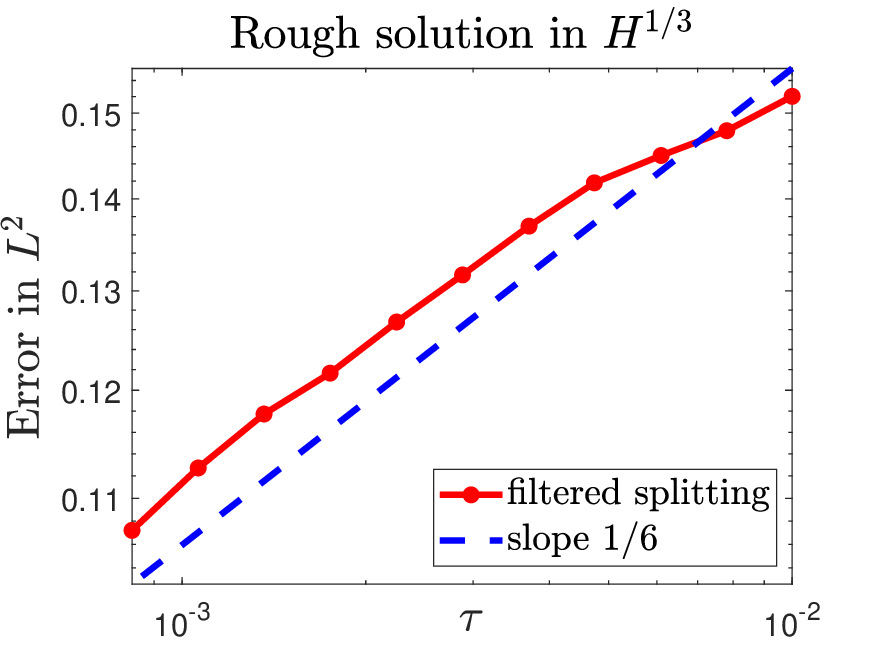}
			\end{minipage}
		}
		\vspace{-4mm}
		\caption{Numerical errors in $L^2$ at the final time $T=0.5$ with rough data in $H^{0.2}$ and $H^{1/3}$, respectively.} \label{1-1}
	\end{figure}
	
	\begin{figure}[htbp]
		\center
		\hspace{-27pt}
		\subfigure{
			\begin{minipage}[c]{0.4\linewidth}
				\centering
				\includegraphics[height=5.2cm, width=1.3\linewidth]{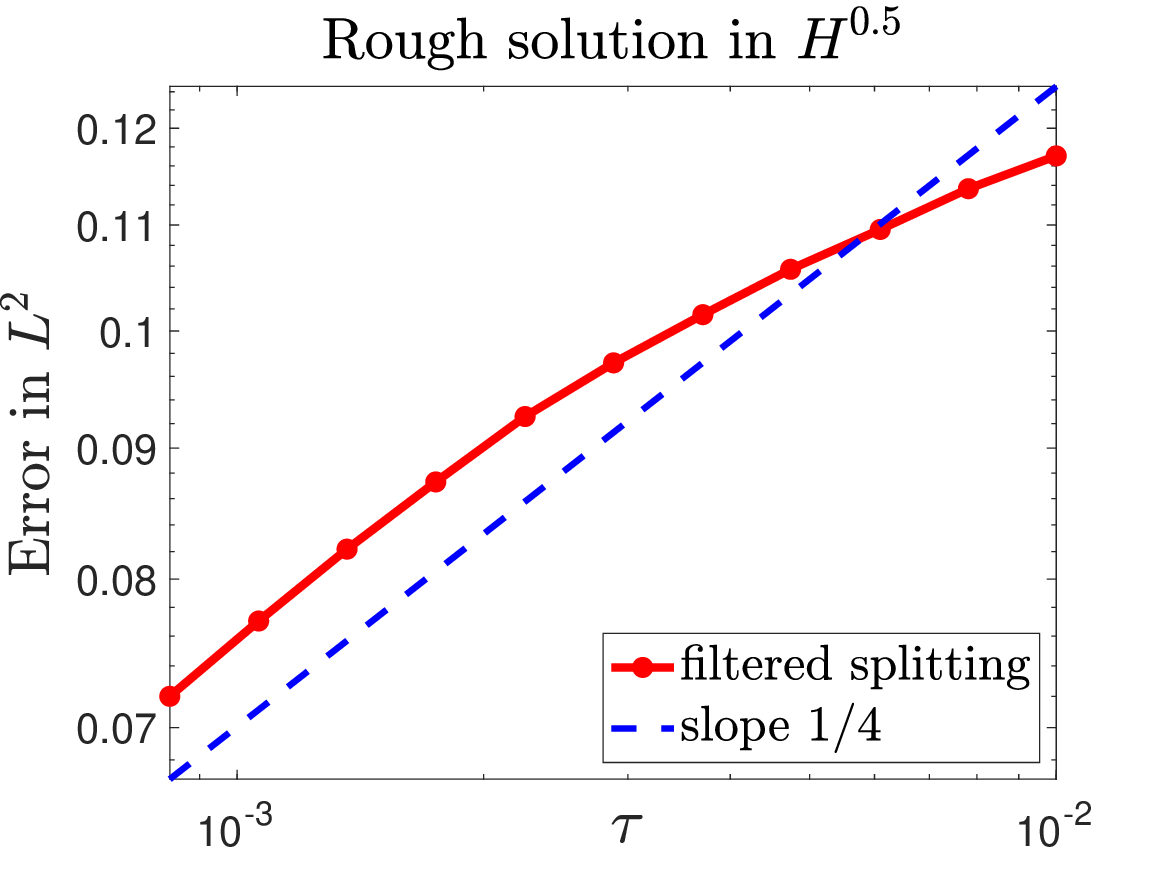}
			\end{minipage}
		}\hspace{26pt}
		\subfigure{
			\begin{minipage}[c]{0.4\linewidth}
				\centering
				\includegraphics[height=5.2cm, width=1.3\linewidth]{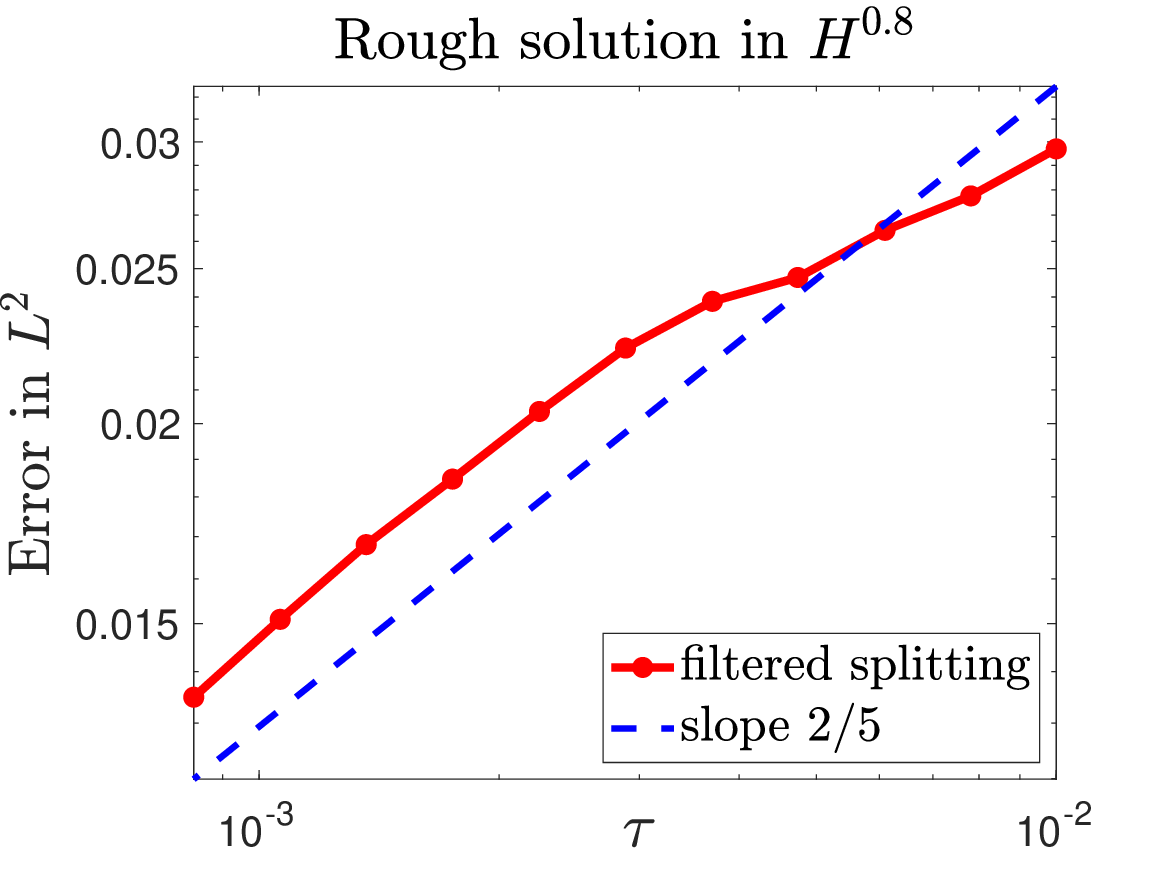}
			\end{minipage}
		}
		\vspace{-4mm}
		\caption{Numerical errors in $L^2$ at the final time $T=0.5$ with rough data in $H^{0.5}$ and $H^{0.8}$, respectively.} \label{1-2}
	\end{figure}

	\bibliographystyle{abbrv}
	
\end{document}